\documentclass[10pt]{amsart}
\usepackage{amsmath}
\usepackage{amssymb}
\usepackage{amsthm}
\usepackage{enumerate}
\numberwithin{equation}{section}

\newtheorem{theorem}{Theorem}[section]

\newtheorem{corollary}[theorem]{Corollary}

\newtheorem{definition}[theorem]{Definition}
\newtheorem{convention}[theorem]{Convention}

\newtheorem{lemma}[theorem]{Lemma}

\newtheorem{proposition}[theorem]{Proposition}
\newtheorem{remark}[theorem]{Remark}

\def\N{\mathbb{N}}
\def\R{\mathbb{R}}
\def\Q{\mathbb{Q}}

\def\AA{\mathcal A}
\def\SS{\mathcal S}

\def\RR{\mathcal R}
\def\FF{\mathcal F}
\def\CC{\mathcal C}
\def\M{\mathcal M}
\def\la{\langle}
\def\ra{\rangle}

\def\nn{\|\cdot\|}
\def\oo{\sqsubset\!\!\!\sqsupset}

\def\ee{\varepsilon}

\def\dd{\delta}

\def\gg{\gamma}

\title{Separable reduction of Fr\'echet subdifferentiability in Asplund spaces}

\author{Marek C\'uth, Mari\'an Fabian}
\address[M.~C\' uth]{Instytut Matematyczny Polskiej Akademii Nauk, \' Sniadeckich 8, 00-656 Warszawa, Poland}
\address[M.~Fabian]{Mathematical Institute of Czech Academy of Sciences, \v Zitn\'a 25, 115 67 Praha 1, Czech Republic}

\email{marek.cuth@gmail.com}
\email{fabian@math.cas.cz}

\thanks{The first author was supported by Warsaw Center of Matemathics and Computer Science (KNOW--MNSzW). 
The second author was supported by grant P201/12/0290 and by RVO: 67985840. }


\date{\today}

\begin{document}

\begin{abstract}
In the framework of Asplund spaces, we use two
equivalent instruments ---rich families and suitable models 
from logic--- for performing separable reductions of various statements on Fr\'echet subdifferentiability of functions.
This way, isometrical results are actually obtained and several existed proofs are substantially simplified.
Everything is based on a new structural characterization of Asplund spaces.
\end{abstract}

\maketitle

\section{Introduction}

Separable reductions for Fr\'echet (sub)diffrerentiability originated some three-four decades ago in works by
D. Gregory \cite[pages 23, 24, 37]{ph}, D. Preiss \cite{pr}, and M. Fabian, N. V. Zhivkov \cite{fz}. In all these papers and in many subsequent ones until the recent contributions ---see \cite{pe}, \cite{i}, \cite{fi1}, \cite{fi2}--- there was a common believe that 
for a successful performing separable reductions of statements involving
Fr\'echet subdifferentiability or differentiability (like non-emptiness of subdifferential, fuzzy calculus, etc.), 
it is necessary to first translate such statements completely
into terms of the Banach space $X$ in question (with no use of its dual $X^*$). 
The present paper destroys this longstanding taboo in the case when we restrict to the framework of Asplund spaces. 
Indeed, once we have at hand a deeper structural characterization
of the Asplund property (see Theorem~\ref{dno}), we can work with the original definition of Fr\'echet (sub)differentiability
(``... there exists an element of \tt the dual $X^*$ \rm such that ...''). 
This way, separable reductions in Asplund spaces can be substantialy simplified and obtained results become exact
(see Theorem~\ref{nova}); for comparison cf. \cite{fi2}.
We believe that such a simplification does not only simplify proofs of existing separable reduction theorems, but it also offers a way how to obtain new results.

Section 2 gathers a partly new background from the world of Asplund spaces.
This is done in the nowadays modern language of rich families.
We believe that a structural result here, Theorem~\ref{dno}, will have applications also beyond the scope of this paper. Section 3 is mostly devoted to the separable reduction of a rather general assertion
involving Fr\'echet subdifferentials (see, Theorem~\ref{nova}). It serves as a common roof for several more concrete 
statements like non-emptiness of Fr\'echet subdifferential, fuzzy calculus, extremal principle, ...
The proof here is based on Theorem~\ref{dno}. 
Section 4 offers a further simplification of arguments, done via a method of suitable models from logic.
Section 5 discusses using rich families versus suitable models. We prove that both methods are in a
sense equivalent in the Asplund spaces and in the spaces 
$C(K)$ of continuous functions
where $K$ is any zero-dimensional compact space; see Theorems \ref{t:CompactRichModely} and \ref{t:richIffModel}. This gives a partial positive answer to \cite[Question 2.8]{ck} and \cite[Question 3.6]{ck}.

\section {Rich familes in Asplund spaces}

 \rm Let $P$ be a set and let $\prec$ be a {\it partial order} on it, i.e. $\prec$ is a subset of 
$P\times P$ which is reflexive, symmetric and transitive, see \cite[page 21]{e}. We agree that, instead of
``$s,t\in\prec$'' we rather write ``$s\prec t$''.
Assume moreover that $P$ is {\it (up)-directed by} $\prec$, i.e., for every $t_1,t_2\in P$ there is
$t_3\in P$ such that $t_1\prec t_3$ and $t_2\prec t_3$. 
A subset $R\subset P$ is called {\it cofinal}/{\it dominating}\,
if for every $t\in P$ there is $r\in P$ such that $t\prec r$. $R$ is called
$\sigma$-{\it complete}/{\it closed} if, whenever $r_1\prec r_2\prec\cdots$ is an increasing sequence in $R$,
then there is $r\in R$ such that $r_i\prec r$ for every $i\in\N$ and $r\prec t$ whenever $t\in P$ and $r_i\prec t$ for every $i\in\N$.
The set $R\subset P$ is called {\it rich}/{\it a club set} if it is both cofinal and $\sigma$-complete.
(Note that the whole $P$ is rich if it is $\sigma$-complete.) 

Now, we are ready to provide concrete examples of the poset $(P,\prec)$ that emerge in the framework of Banach spaces.
Let $Z$ be a (rather non-separable) Banach space. By $\SS(Z)$
we denote the family of all separable closed subspaces of $Z$ and we endow it by the partial order
``$\subset$''. Thus, we can consider \it rich families \rm in the poset $(\SS(Z),\subset)$. 
We then also simply say that they are \it rich in \rm $Z$. This example of rich family was first articulated
in a paper \cite{bm} by J.M. Borwein and J. Zhu.
The power of rich families is demonstrated by the following fundamental fact (see \cite{bm} and 
also \cite[page 37]{lpt}).

\begin{proposition}\label{bm} 
The intersection of two (even of countably many) rich families of a given Banach space is
(not only non-empty but again even) rich.
\end{proposition}

Let $k\in\N$ be greater that $1$, and let $X_1,\ldots, X_k$ be Banach spaces. 
By a \it block \rm we understand any product $Y_1\times\cdots\times Y_k$.
The symbol $\SS_{{\oo}^{\!\!\!\!\!\oo}}(X_1\times\cdots\times X_k)$ will denote the (rich)
family of all blocks $Y_1\times\cdots\times Y_k$ such that $Y_1\in\SS(X_1), \ \ldots,\ 
Y_k\in\SS(X_k)$. Any subset of $\SS_{{\oo}^{\!\!\!\!\!\oo}}(X_1\times\cdots\times X_k)$
will be called a \it block-family \rm in $\SS(X_1\times\cdots\times X_k)$ or just in 
$X_1\times\cdots\times X_k$. If $k=2$, we speak about \it rectangles \rm and \it rectangle-families \rm and
we write $\SS_{\oo}(X_1\!\times\! X_2)$ instead of $\SS_{{\oo}^{\!\!\!\!\!\oo}}(X_1\!\times\! X_2)$.  For a Banach space $X$, with dual $X^*$, we will frequently
work with rectangle-families in $\SS_{\oo}(X\times X^*)$. 

We conclude by one warning. If $\RR$ is a rich rectangle-family in $\SS_{\oo}(X\times X^*)$
, we do not know if, the ``projection'' of it on, say, the second coordinate, that is, the family
$\big\{Y:\ V\times Y\in\RR$ for some $V\in \SS(X)\big\}$, is rich in $\SS(X^*)$. 
Fortunately, in one important case, the ``projection'' of $\RR$ to the first coordinate is again rich;
see Theorems~\ref{dno} a~\ref{nova} below.

Let $X$ be a Banach space. If $A\subset X$, the symbols $\overline{\rm sp}\,A$ 
and ${\rm sp}_\Q A$ mean the closed linear span of $A$ and the set consisting of all finite linear combinations of
elements in $A$ with rational coefficients, respectively. For $ A\subset X$ and $B\subset X^*$ we put
$
B|_A:= \big\{x^*{}|_A:\ x^*\in B\big\};
$
hence, if $A$ is a subspace of $X$, then $B|_A$ is a subset of the dual space $A^*$.
Let $\CC(X)$ and $\CC(X^*)$ denote the families of all countable subsets of $X$ and $X^*$ respectively.

\rm A Banach space is called \it Asplund \rm if every convex continuous function on it is Fr\'echet differentiable
at a point (equivalently, at the points of a dense set, yet equivalently, at the points of a dense $G_\delta$ set).
An important, and widely used, equivalent condition for the Asplund property of a Banach space is
that \sl every separable subspace of it has separable dual, \rm see \cite[Theorem 2.34]{ph}. 

Next, we introduce a concept which serves as a link between $X$ and $X^*$ (and exists right
if and only if $X$ is Asplund).

\begin{definition}\label{generator}
\rm By an \it Asplund generator
in a Banach space $X$ we understand any correspondence 
$G:\CC(X)\longrightarrow\CC(X^*)$ such that
\smallskip

\noindent (a) $\big(\overline{\rm sp}\,C\big)^*\subset \overline {G(C)|_{\overline{\rm sp}\,C}}\ $ for every $C\in\CC(X)$;

\smallskip
\noindent (b) if $C_1,\ C_2,\ \ldots$ is an increasing sequence in $\CC(X)$, then
$G(C_1\cup C_2\cup\cdots)=G(C_1)\cup G(C_2)\cup\cdots$; 

\smallskip
\noindent (c) $\bigcup\{G(C):\ C\in\CC(X)\}$ is a dense subset in $X^*$; and\smallskip

\noindent (d) if $C_1, C_2\in\CC(X)$ are such that $\overline{\rm sp}\, C_1=\overline{\rm sp}\, C_2$,
then $\overline{\rm sp}\, G(C_1)=\overline{\rm sp}\, G(C_2)$.

\end{definition}

Now, consider a Banach space $(X,\|\cdot\|)$, not being Hilbert. Let $V$ be a (closed) subspace of it.
We focus on a question whether the dual $V^*$ is (linearly isometric to) a subspace of $X^*$. 
(If there were so, then some arguments in the next sections would become quite easy.) There do exist
situations when this is so. For instance, if $X$ is $c_0(\Gamma)$ or $\ell_p(\Gamma),\ 1\le p<+\infty$, and
$V$ is $c_0(N)$ or $\ell_p(N)$, where $N\subset \Gamma$. However, 
we are afraid that for general $X$ and $V$ this may not be so.
One way how to look for well behaving $V$'s is to consider Hahn-Banach extensions of elements of $V^*$:
For every $v^*\in V^*$ there are $x^*\in X^*$ such that $x^*{}|_V=v^*$ and $\|v^*\|=\|x^*\|$.
Such an assignment is usually multivalued. Moreover, there is no guarantee that there
exists a selection of the assignment which would be linear. Fortunately, and this is the content of the next structural theorem:
If $X$ is Asplund, there are \tt plenty \rm of very well behaving $V$'s.
We think that this statement actually elucidates what the Asplund property of a Banach space is. The proof of it
gathers together ideas from several papers ranging over half a century: 
\cite{cf, st, fg, g, jz, t, l} (quoted in reverse chronological order).

\begin{theorem}\label{dno} {\rm (Main)} Let $(X,\nn)$ be a (rather non-separable) Banach space. Then 
the following assertions are mutually equivalent.

\item{(i)} $X$ is an Asplund space.

\item{(ii)} $X$ admits an Asplund generator.

\item{(iii)} There exists a rich rectangle-family $\mathcal A\subset \mathcal S_{\oo}(X\times X^*)$ such that
$Y_1\subset Y_2$ whenever $V_1\times Y_1,\ V_2\times Y_2$ are in $\AA$ and $V_1\subset V_2$, and
for every $V\times Y\in\AA$ the assignment
$
Y\ni x^*\longmapsto x^*{}|_{V}\in V^*
$
is a surjective isometry.

\item{(iv)} There exists a  cofinal rectangle-family $\mathcal A\subset \mathcal S_{\oo}(X\times X^*)$ such that
for every $V\times Y\in\AA$ the assignment $Y\ni x^*\longmapsto x^*{}|_{V}\in V^*$
is a surjection.

\end{theorem}

\begin{proof} (i)$\Longrightarrow$(ii). In order not to get lost in the case of general Asplund space, 
assume first that the norm $\nn$ on $X$ is Fr\'echet smooth, or more generally, 
that there exists a smooth function $f:X\rightarrow\R$, with continuous derivative $f'$ 
such that $f'(V)|_V$ is dense in $V^*$ for every subspace $V$ of $X$; note that
this easily implies that $X$ is Asplund. Define then $G:\CC(X)\longrightarrow\CC(X^*)$ by
$$
\CC(X)\ni C\longmapsto f'\big(\,{\rm sp}_\Q C\big)=:G(C)\in\CC(X^*).
$$
It remains to verify the properties (a), (b), (c), and (d) in Definition~\ref{generator}. As regards (a), 
fix any $C\in\CC(X)$ and any non-zero $v^*$ in $(\,\overline{\rm sp}\, C)^*$. 
Let any $\ee>0$ be given. The properties of $f$ provide a 
$v\in {\rm sp}_\Q\, C$ such that $\big\|v^*-f'(v)|_{\overline {\rm sp}\,C}\big\|<\ee $. But 
$f'(v)$ belongs to $G(C)$. And, as $\ee>0$ was arbitrary, we get that $v^*$ belongs 
$\overline{G(C)|_{\,\overline {\rm sp}\,C}}\,$. Thus (a) is verified.
 As regards (b), let $C_1, C_2,\ldots$ be as in the premise. Because our Asplund generator $G$ is ``monotone'', it is enough to prove the
inclusion ``$\subset$''. So, pick any $x^*$ in $G(C_1\cup C_2\cup\cdots)$. Since $C_1\subset C_2\subset\cdots$, there is
$m\in\N$ so big that $x^*$ belongs to $G(C_m)$. We thus verified (b).
The claim (c) follows immediately from the fact that $f'(X)$ is dense in $X^*$ and from the
definition of $G$. The last property (d) is guaranteed by the continuity of $f'$.

If we are facing a general Asplund space (and we do not have at hand the function $f$ as above),
we have to work harder. Either, we use \cite[Propositions 1 and 2]{cf}, based on Ch. Stegall's ideas (and
proved without use of Simons' lemma),
or we exploit an information from \cite{fg} (where Simons' lemma is needed!); see also \cite[Remark 2]{cf}.
More concretely, using symbols $\mathcal L$ and $\Lambda$ from \cite{cf}, define the Asplund generator $G:\CC(X)\longrightarrow\CC(X^*)$ by
$$
\CC(X)\ni C\longmapsto \Lambda\big(\mathcal L\big({\rm sp}_\Q\, C\cap B_X\big)\big)=:G(C)\in \CC(X^*).
$$
Now (a) in Definition~\ref{generator} follows from \cite[Proposition 1]{cf} and the proof of it.
(We actually get a stronger inclusion that $\big(\overline{\rm sp}\,C\big)^*\subset \overline {G(C)}\,|\,_{\overline{\rm sp}\,C}$.)
(b) follows immediately from the very definition of our $G$, the definition of $\Lambda,\ \mathcal L$, and 
from the monotonicity of the sequence $C_1, C_2,\ldots$ (c) follows immediately from \cite[Proposition 1]{cf}.
(d) follows easily from the properties of $\mathcal L$ and $\Lambda$,
and from the definition of $G$. We thus proved that (ii) holds in a general Asplund space.
\smallskip

(ii)$\Longrightarrow$(iii). Let $G:\CC(X)\longrightarrow\CC(X^*)$ be a generator in $X$. Define $\AA\subset\SS_{\oo}(X\times X^*)$
as the family consisting of all rectangles $\overline{\rm sp}\, C\times \overline{\rm sp}\, G(C)$,
with $C\in\CC(X)$, such that the assignment
\begin{equation}\label{si}
\overline{\rm sp}\, G(C)\ni x^*\longmapsto x^*{}|\,_{\overline{\rm sp}\, C}\in(\overline{\rm sp}\, C)^* 
\end{equation}
is a surjective isometry. We shall show that $\AA$ is a rich family. 

As regards the cofinality of $\AA$, fix any $V\times Y\in \SS_{\oo}(X\times X^*)$. 
Since $G$ is an Asplund generator, the condition (c) guarantees that there is $C_0\in\CC(X)$ so big that
$\overline{C_0}\supset V$ and $\overline{G(C_0)}\supset Y$. Assume that for some $m\in\N$
we already found countable sets $C_0\subset C_1\subset\cdots\subset C_{m-1}\subset X$.
Realizing that ${\rm sp}_\Q\, G(C_{m-1})$ is countable, we find $C_m\in\SS(X)$ so big that
$C_m\supset C_{m-1}$ and that $\|x^*\|=\sup\,\langle x^*,C_m\cap B_X\rangle$ for every $x^*\in{\rm sp}_\Q\, G(C_{m-1})$. Do so for every $m\in\N$ and put finally $C:=C_0\cup C_1\cup\cdots$. Clearly
$C\in\CC(X)$ and also $\overline{\rm sp}\,C\times\overline{\rm sp}\,G(C)\supset V\times Y$.
It remains to show that the assignment (\ref{si}), with our just constructed $C$, is a
surjective isometry. 

Take any $x^*\in{\rm sp}_\Q\, G(C)$.
Since $x^*$ is a linear combination of finitely many elements from $G(C)$ and that
$C_0\subset C_1\subset\cdots$, the property (b) of $G$ provides an $m\in\N$ so big that
$x^*$ belongs to ${\rm sp}_\Q\, G(C_{m-1})$. But then, from the construction above,
$$
\|x^*\|=\sup\,\langle x^*,C_m\cap B_X\rangle \le \sup \langle x^*,\overline{\rm sp}\, C\cap B_X\rangle
=\big\|x^*|\,_{\overline{\rm sp}\, C}\big\| \le \|x^*\|.
$$
And, as $\overline{\rm sp}\, G(C)=\overline{{\rm sp}_\Q\, G(C)}$, we get that $\big\|x^*|\,_{\overline{\rm sp}\,C}\big\|=\|x^*\|$
for every $x^*\in \overline{\rm sp}\, G(C)$. We proved that the assignment (\ref{si}) with our $C$ is isometrical.

Now, fix any $v^*\in(\overline{\rm sp}\,C)^*$. By (a) from Definition~\ref{si}, 
there is a sequence $(x^*_n)$ in $G(C)$ so that $\big\|v^*-x^*_n|\,_{\overline{\rm sp}\,C}\big\|\longrightarrow0$
as $n\to\infty$. By the isometric property of (\ref{si}) just proved, we have that $\|x^*_i-x^*_j\|=
\big\|x^*_i|\,_{\overline{\rm sp}\, C}-x^*_j|\,_{\overline{\rm sp}\, C}\big\|\longrightarrow0$ as $i,j\to\infty$.
Put $x^*:=\lim_{n\to\infty}x^*_n$; then $x^*\in\overline{G(C)}\subset \overline{\rm sp}\,G(C)$ and
$v^*=x^*|\,_{\overline{\rm sp}\, C}$. This shows the surjectivity of the assignment (\ref{si})
with our $C$. 
This way, we proved that $\overline{\rm sp}\,  C\times \overline{\rm sp}\,  G(C)$
belongs to $\AA$, and hence, the family $\AA$ is cofinal.

For checking the $\sigma$-completeness of $\AA$, consider any increasing sequence 
$V_1\times Y_1,\ V_2\times Y_2,\ \ldots$ of elements in $\AA$. Then, clearly,
$\overline{V_1\times Y_1\cup V_2\times Y_2\cup\cdots}\,$ is of form $V\times Y$
and this is an element of $\SS_{\oo}(X\times X^*)$. Also, clearly,
$V=\overline{V_1\cup V_2\cup\cdots}$ and $Y=\overline{Y_1\cup Y_2\cup\cdots}\,$. From the definition
of $\AA$, for every $i\in\N$ find $C_i\in\CC(X)$ such that $V_i=\overline{\rm sp}\, C_i$
and $Y_i=\overline{\rm sp}\, G(C_i)$. Put $C:=C_1\cup C_2\cup\cdots$; then $C\in\CC(X)$. Since $V_1\subset V_2\subset\cdots$ and $Y_1\subset Y_2\subset\cdots$, 
some rather boring reasoning, profiting from the properties (b) and (d) of $G$ in Definition~\ref{si},
guarantees that $V=\overline{\rm sp}\, C$ and $Y=\overline{\rm sp}\, G(C)$. (Hint: Replace the sequence
$C_1, C_2, \ldots$ by the increasing one $C_1,\ C_1\cup C_2,\ C_1\cup C_2\cup C_3,\ \ldots$)
Hence, by (a), $V^*\subset\overline{Y|_V}$. 

Now, we will verify that the assignment $Y\ni x^*\longmapsto x^*|_V\in V^*$ is a surjective isometry. 
As regards the isometric property, we recall that for every $i\in\N$ the rectangle $V_i\times Y_i$ belongs to
$\AA$, and so for every $x^*\in Y_i$ we have
$$
\|x^*\|=\big\|x^*|_{V_i}\big\| \le \big\|x^*|_V\big\| \le \|x^*\|.
$$
It then follows, using the density of $Y_1\cup Y_2\cup\cdots$ in $Y$, that $\|x^*\|=\big\|x^*|_V\big\|$
for every $x^*\in Y$. 
Now, once having the information just proved, we have that $\overline{Y|_V}=Y|_V\ (\subset V^*)$, and hence
$V^*=Y|_V$. Therefore, summarizing all the above, 
we are sure that our $\AA$ is a rich family.

Finally, consider any $V_1\times Y_1,\ V_2\times Y_2$ in $\AA$
such that $V_1\subset V_2$. From the very definition of $\AA$ we find $C_1, C_2\in\CC(X)$ such that
$\overline{\rm sp}\, C_1= V_1$ and $\overline{\rm sp}\, C_2= V_2$. Then
$$
C_2\subset C_1\cup C_2\subset \overline{\rm sp}\,  C_1\cup \overline{\rm sp}\,  C_2 = V_1\cup V_2=V_2 = \overline{\rm sp}\,  C_2,
$$
and so $\overline{\rm sp}\, C_2\subset \overline{\rm sp}\, (C_1\cup C_2)\subset \overline{\rm sp}\,  C_2$.
Now (d) in Definition ~\ref{generator} gives that $\overline{\rm sp}\, G(C_2)=\overline{\rm sp}\, G(C_1\cup C_2)$, and so
$$
Y_2=\overline{\rm sp}\, G(C_1\cup C_2) \stackrel{(b)}= \overline{\rm sp}\, \big(G(C_1)\cup G(C_1\cup C_2)\big)
\supset \overline{\rm sp}\, G(C_1)=Y_1.
$$
We completely proved (iii).

\smallskip
(iii)$\Longrightarrow$(iv) is trivial.
\smallskip

(iv)$\Longrightarrow$(i). Assume (iv) holds. Let $Z\in\SS(X)$ be arbitrary. From the cofinality of $\AA$,
find $V\times Y\in\AA$ such that $V\times Y\supset Z\times\{0\}$. Then $V^*$, being the image of $Y\,(\in\SS(X^*))$, 
is itself separable. It then follows that $Z^*$, the quotient of $V^*$, must be also separable. Now it remains to use the
aforementioned characterization of the Asplund property, and thus (i) follows.
\end{proof}

\begin{remark}\label{2.4}
\rm Assume that the norm $\nn$ on $X$ is Fr\'echet smooth and define $f:=\|\cdot\|^2$.
Then for every subspace $V\subset X$ we get that $V^*\subset \overline{f'(V)|_V}$ but not
$V^*\subset\overline{f'(V)}\,|_V$. Indeed, this stronger inclusion seems to be a privilege of only some $V$'s;
we can find such subspaces by playing a suitable ``volleyball'' with countably many moves, see the proof of
(ii)$\Rightarrow$(iii) above.
(Fortunately, these ``selected/better'' $V$'s form a rich family in $\SS(X)$.) From this, and from the proof of implication (i)$\Rightarrow$(ii) above,
it follows that using the Stegall's approach here is somehow stronger and simpler, see \cite[Proposition 1]{cf}. Likewise, the Stegall's approach is stronger and simpler than that from \cite{fg}, see \cite[Remark 2]{cf}.
\end{remark}

It can be useful to extend Theorem~\ref{dno} to the following statement.

\begin{theorem}\label{dnoo}
Let $(Z,\nn)$ be a Banach space, $(X,\nn)$ an Asplund space, $T: Z\rightarrow X$ a bounded linear operator,
and let $z^*\in Z^*$ be given.
Then there exists a rich block-family $\AA_T$ in $Z\times X\times X^*$ such that 
$Y_1\subset Y_2$ whenever $U_1\!\times\! V_1\!\times\! Y_1,\ U_2\!\times\! V_2\!\times\! Y_2\in{\mathcal A}_T$ and
$V_1 \subset V_2$, and that
for every $U\!\times\! V\!\times\! Y$ in $\AA_T$ we have $T(U)\subset V$, the restricion assignment
$Y\ni x^*\longmapsto x^*|_V\in V^*$ is a surjective isometry, and $\|T^*x^*-z^*\|=\big\|(T|_U)^*(x^*|_V)-(z^*|_U)\big\|$
for every $x^*\in Y$.
\end{theorem}

\begin{proof}
It is easy (and left to a reader) to check that the rectangle-family $\RR_T$ consisting of all $U\times V\in\SS_{\oo}(Z\times X)$
such that $T(U)\subset V$ is rich in $Z\times X$. Denote 
$$
\RR_1:=\big\{U\!\times\! V\!\times\! Y:\ U\!\times V\in\RR_T\ \ {\rm and}\ \ Y\in\SS(X^*)\big\},
$$$$
\RR_2:=\big\{U\!\times\! V\!\times\! Y:\ U\in\SS(Z)\ \ {\rm and}\ \ V\times Y\in\AA\big\}
$$
where $\AA$ is from Theorem~\ref{dno}. Clearly, both these families are rich, and therefore
$\RR:=\RR_1\cap\RR_2$ is a rich block-family in $\SS_{{\oo}^{\!\!\!\!\oo}}(Z\!\times\! X\!\times X^*)$.
Clearly, every triple $U\!\times\! V\!\times\! Y$ in $\RR$ possesses the first two properties from
the conclusion of our theorem. Now, define the family 
$$
\AA_T:=\big\{U\!\times\! V\!\times\! Y\in\RR:\ \|T^*x^*-z^*\|=\big\|(T|_U)^*(x^*|_V)-(z^*|_U)\big\|\ \ \hbox{for every}\ \  x^*\in Y\big\}.
$$
Clearly, $\AA_T$ has all the three required properties. Thus, it remains to check that $\AA_T$ is rich.

As regards the cofinality of $\AA_T$, consider any $M\in\SS(Z\times X\times X^*)$. From the cofinality
of $\RR$, find $U_0\!\times\! V_0\!\times\! Y_0$ in $\RR$ such that $U_0\!\times\! V_0\!\times\! Y_0\supset M$.
We shall construct an increasing sequence $U_m\!\times\! V_m\!\times\! Y_m,\ m\in\N$, in $\RR$ as follows.
Let $m\in\N$ and assume that we have already found $U_{m-1}\!\times\! V_{m-1}\!\times\! Y_{m-1}$. 
Using the separability of $Y_{m-1}$ find $C_{m-1}\subset\CC(Z)$ such that ${\overline {C_{m-1}}}\supset U_{m-1}$ and
$\|T^*x^*-z^*\|=\sup\,\la T^*x^*-z^*,C_{m-1}\cap B_Z\ra$ for every $x^*\in Y_{m-1}$. Find 
$U_m\!\times\! V_m\!\times\! Y_m$ in $\RR$ so big that it contains $(U_{m-1}\cup C_{m-1})\!\times\! V_{m-1}\!\times\! Y_{m-1}$. 
Doing so for every $m\in\N$,
put finally $U:=\overline{\bigcup U_m}\,, \
V:=\overline{\bigcup V_m}\,$, and $Y:=\overline{\bigcup Y_m}\,$. Clearly, $U\!\times\! V\!\times\! Y=
\overline{\bigcup U_m\!\times\! V_m\!\times\! Y_m}\, \supset M$.
The $\sigma$-completeness of $\RR$ guarantees that $U\!\times\! V\!\times\! Y$ lies in $\RR$.
Now fix any $m\in\N$ and any $x^*\in Y_{m-1}$. We can estimate
\begin{eqnarray*}
\|T^*x^*-z^*\|&=&\sup\big\la    T^*x^*-z^*,C_{m-1}\cap B_Z\big\ra \le \sup\big\la T^*x^*-z^*, B_U\big\ra\\
&=&\sup\big\{\big\la x^*|_V,\big(T|_U\big)u\big\ra-\la z^*|_U,u\ra:\ u\in B_U\big\}\\ 
&=&  \big\|(T|_U)^*(x^*|_V)-z^*|_U\big\| \le \|T^*x^*-(z^*)\|\,.
\end{eqnarray*}
Thus $\|T^*x^*-z^*\| = \big\|(T|_U)^*(x^*|_V)-(z^*)|_U\big\|$ for every $x^*$ from $\bigcup Y_m$, and finally, for every 
$x^*$ from $Y$. We verified that $U\times V\times Y\in{\mathcal A}_T$, and hence ${\mathcal A}_T$ is cofinal.

As regards the $\sigma$-completeness of $\AA_T$, consider any increasing sequence $U_1\!\times\! V_1\!\times\! Y_1,\ 
U_2\!\times\! V_2\!\times\! Y_2,\ldots$ in $\AA_T$. Put $U:=\overline{\bigcup U_i}\,, \
V:=\overline{\bigcup V_i}\,$, and $Y:=\overline{\bigcup Y_i}\,$. Clearly, $U\!\times\! V\!\times\! Y=
\overline{\bigcup U_i\!\times\! V_i\!\times\! Y_i}\,$. As $\RR$ was rich, our $U\!\times\! V\!\times\! Y$ belongs to it.
Take any $i\in\N$ and any $x^*\in Y_i$. 
Since $U_i\times V_i\times Y_i\in{\mathcal A}_T$, we have that $\|T^*x^*-z^*\|=
\big\|(T|_{U_i})^*(x^*|_{V_i})-(z^*|_{U_i})\big\|$. But we can easily verify the following monotonicity
$$
\big\|(T|_{U_i})^*(x^*|_{V_i})-(z^*|_{U_i})\big\| \le \big\|(T|_U)^*(x^*|_V)-(z^*|_U)\big\| \le \|T^*x^*-z^*\|.
$$
Thus $\|T^*x^*-z^*\|=\big\|(T|_U)^*(x^*|_V)-(z^*|_U)\big\|$ holds for every $x^*$ from $\bigcup Y_i$, 
and finally for every $x^*$ from $Y$. We proved that $U\times V\times Y$ belongs to $\AA_T$,
and therefore this family is $\sigma$-complete.

\end{proof}

\begin{remark}\label{2.6}
{\em Of course, Theorem~\ref{dnoo} can be easily extended to several spaces $Z_1,\ldots,Z_k$, to $z^*_i\in Z^*_i$, and to operators
$T_i:Z_i\rightarrow X,\ i=1,\ldots,k$.}
\end{remark}

\section{Separable reduction for statements with Fr\'echet subdifferentials}

Let $(X,\|\cdot\|)$ be a Banach space, let $f:X\longrightarrow(-\infty,+\infty]$ be any proper fnction, i.e. $f\not\equiv+\infty$, and let $x\in X$ be a point where $f(x)<+\infty$. The {\em Fr\'echet subdifferential} $\partial_Ff(x)$ of $f$ at $x$
is the (possibly empty) set consisting of all $x^*\in X^*$ such that 
$f(x+h)-f(x)-\la x^*,h\ra > -o(\|h\|)$ for all $0\neq h\in X$ where $o:(0,+\infty)\longrightarrow
[0,+\infty]$ is a suitable function with the property that $\frac{o(t)}t\rightarrow0$ as $t\downarrow0$; or in other words, if for every $\ee>0$ there is $\delta>0$ such that
$\frac1{\|h\|}\big(f(x+h)-f(x)-\la x^*,h\ra\big)> -\ee$ whenever $h\in X$ and $0<\|h\|<\delta$.

The novelty presented below is that, under
the (small) price of \tt restricting to  Asplund spaces, \rm
for separable reductions of statements involving
Fr\'echet subdifferentials, we do not need to translate first these statements into the terms of the primal space $X$
(as it was used to do in the last three decades). 
This is a noticeable simplification when comparing with the so far existing separable reductions; see \cite{fi2}.
In addition we get ``isometric'' statements, which
substantially improve those from \cite{fi2}, and moreover unable 
to simplify other proofs from the quoted paper.

\begin{theorem}\label{nova}
{\rm (Main)} Let $(X,\nn)$ be a (rather non-separable) {\tt Asplund} space and let $f:X\longrightarrow(-\infty,+\infty]$ be any proper function. 
Then there exists a rich rectangle-family $\RR\subset\SS_{\oo}(X\times X^*)$ such that 
$Y_1\subset Y_2$ whenever $V_1\times Y_1,\ V_2\times Y_2\in\RR$ and $V_1\subset V_2$,
with further properties  
that for every $V\times Y\in\RR$ the assignment
$Y\ni x^*\longmapsto x^*|_V\in V^*$ is an {\tt isometry} from $Y|_V$ onto $V^*$ 
and for every $v\in V$ we have that 
$$
\big(\partial_Ff(v)\cap Y\big)|_V=\big(\partial_Ff(v)\big)|_V=\partial_F(f|_V)(v)\,.
$$
\end{theorem}

\begin{remark} {\em It is worth to compare the theorem above with what was proved in \cite{fz,f3}: \sl
In a general (possibly non-Asplund) Banach space $X$, there was found a cofinal family  $\mathcal F$ in $\SS(X)$ such that for every $V\in\mathcal F$ and for
every $v\in V$ the non-emptiness of $\partial_F(f|_V)(v)$ implied the non-emptiness of $\partial_Ff(v)$.}
\end{remark} 

\begin{proof} We obviously have that 
$$
\big(\partial_Ff(v)\cap Y\big)|_V \subset \big(\partial_Ff(v)\big)|_V \subset \partial_F(f|_V)(v).
$$
It remains to prove that $\partial_F(f|_V)(v) \subset \big(\partial_Ff(v)\cap Y\big)|_V$ holds for every $v\in V$.
For $x\in X,\ x^*\in X^*,\ r\in\R,\ 0<\dd_1<\dd_2$, and $V\subset X$ we define
$$
I_V(x,x^*,r,\delta_1,\delta_2):=\inf\bigg\{\frac1{\|h\|}\big(f(x+h)-r-\langle x^*,h\rangle\big):\
h\in V\ \ {\rm and}\ \ \dd_1<\|h\|<\dd_2\bigg\};
$$ 
if $V=X$, we omit the index $V$. (A novelty here is that we operate with $x^*$, element of the dual $X^*$,
which was ``forbidden'' for three decades, and also that $f(x)$ is replaced by $r\in\R$.) Further for each such cortege $x, x^*, r, \dd_1, \dd_2$ and each $\gg>0$,  
if $I(x, x^*, r, \dd_1, \dd_2)>-\infty$, we find a vector $h(x, x^*, r, \dd_1, \dd_2,\gg)\in X$ such that
\begin{equation}\label{16}
\frac1{\|h(x, x^*, r, \dd_1, \dd_2,\gg)\|}\big(f(x+h(x, x^*, r, \dd_1, \dd_2,\gg))-r-\langle x^*,h(x, x^*, r, \dd_1, \dd_2,\gg)\rangle\big)
<I(x, x^*, r, \dd_1, \dd_2) + \gamma
\end{equation}

Let $\AA\subset\SS_{\oo}(X\times X^*)$ be the rich family found in Theorem~\ref{dno}. We define a family $\RR$
as that consisting of all $V\times Y\in\AA$ satisfying
\begin{equation}\label{17}
I(x,x^*,r,\dd_1,\dd_2)=I_V(x,x^*,r,\dd_1,\dd_2)\ \ {\rm whenever}\ \ x\in V,\ x^*\in Y,\ r\in\R,\ \ {\rm and}\ \ 0<\dd_1<\dd_2.
\end{equation}
We shall prove that $\RR$ is cofinal in $\SS(X\times X^*)$. Let $\Q$ denote the set of all rational numbers and put
$\Q_+=\Q\cap(0,+\infty)$. Fix any $Z\in \SS(X\times X^*)$.
Since $\AA$ is rich, there is $V_0\times Y_0\in \AA$ such that $V_0\times Y_0\supset Z$.
Find countable sets $C_0, D_0$ contained and dense in $V_0$ and $Y_0$, respectively.
We shall construct increasing sequences $Y_0\times V_0,\ V_1\times Y_1,\ V_2\times Y_2,\ \ldots$
in $\AA$, and $C_0\times D_0,\ C_1\times D_1,\ C_2\times D_2,\ \ldots$ in $\CC_{\oo}(X\times X^*)$
such that $\overline{C_i}=V_i,\ \overline{D_i}=Y_i$ for every $i\in\N$, and having some extra properties
described below. Let $m\in\N$ be arbitrary and assume that we have already found $V_{m-1}, Y_{m-1}, C_{m-1}, D_{m-1}$.
From the cofinality of $\AA$ we find $V_m\times Y_m\in\AA$ such that 
$V_m$ contains the (countable) set   
$$
\widetilde C:=C_{m-1}\cup \big\{h(x, x^*, q, \dd_1, \dd_2,\gg):\ x\in C_{m-1},\  x^*\in D_{m-1},\
q\in\Q,\ \dd_1,\dd_2,\gamma\in\Q_+,\ {\rm and}\ \dd_1<\dd_2\big\}
$$
and $Y_m\supset Y_{m-1}$. Find then a countable set $\widetilde{C}\subset C_m\subset V_m$
such that $\overline{C_m}=V_m$ and a countable set $D_{m-1}\subset D_m\subset Y_m$ so that $\overline{D_m}=Y_m$.
Do so subsequently for every $m\in\N$. Put $V:=\overline{V_0\cup V_1\cup V_2\cup\cdots}\,$ and 
$Y:=\overline{Y_0\cup Y_1\cup Y_2\cup\cdots}\,$. The $\sigma$-completeness of $\AA$ guarantees that
$V\times Y$ belongs to $\RR$.

We shall show that $V\times Y\in\RR$. This means that we have to verify (\ref{17}). 
So, fix any cortege $x,x^*,r,\dd_1,\dd_2$ as there. Consider any $h\in X$ such that $\dd_1<\|h\|<\dd_2$.
We have to show that  $\frac1{\|h\|}(f(x+h)-r-\langle x^*,h\rangle) \ge I_V(x,x^*,r,\dd_1,\dd_2)$.
This inequality is trivially satisfied if $I_V(x,x^*,r,\dd_1,\dd_2)=-\infty$ Further assume that this is not so.
Pick some $\dd_1',\dd_2'\in\Q$ such that $\dd_1<\dd_1'<\|h\|<\dd_2'<\dd_2$. 
It is easy to check that $V=\overline{C_0\cup C_1\cup\cdots}\, $ and $Y=\overline{D_0\cup D_1\cup\cdots}\,$.
Find $x_0\in C_0,\ x_1\in C_1, \ \ldots$ and $x^*_0\in D_0,\ x^*_1\in D_1,\ \ldots $ such that $\|x_i-x\|\longrightarrow0$
and $\|x^*_i-x^*\|\longrightarrow0$ as $i\rightarrow\infty$. Consider any fixed $\gg\in\Q_+\,$. 
Pick $q\in\Q$ such that $|q-r|<\gamma\|h\|$. Denote 
$N_1:= \big\{i\in \N:\ \|x_i-x\|<\min\{\dd_1'-\dd_1,\dd_2-\dd_2'\}\big\}$;
this is a co-finite set in $\N$. Now, take any $k\in V$, with $\dd_1'<\|k\|<\dd_2'$. 
For $i\in N_1$ we have $\dd_1<\|x_i-x+k\|<\dd_2$ and then we can estimate
\begin{eqnarray}\label{81}
 \begin{aligned}
  &\frac1{\|k\|}\big(f(x_i+k)-q-\la x^*_i,k\ra\big)\\ \
  &=\frac{\|k+x_i-x\|}{\|k\|}\cdot\frac1{\|k+x_i-x\|}\big(f(x+(x_i-x+k))-r-\la x^*,x_i-x+k\ra\big)\\ 
  &+\frac1{\|k\|}\big(\la x^*,x_i-x+k\ra - \la x^*_i,k\ra\big) + \frac{r-q}{\|k\|}\\ 
  &\ge \Big(1+s_i\frac{\|x_i-x\|}{\|k\|}\Big)I_V(x,x^*,r,\dd_1,\dd_2)-\frac1{\dd_1}\big(\|x^*\|\|x_i-x\|  +\dd_2\|x^*-x^*_i\|\big) -\gg\frac{\dd_2}{\dd'_1} 
 \end{aligned}
\end{eqnarray}
where $s_i=1$ if $I_V(x,x^*,r,\dd_1,\dd_2)\le0$ and $s_i=-1$ otherwise.
It then follows that
\begin{eqnarray}\label{36}
 \begin{aligned}
  &I_V(x_i,x^*_i,q,\dd_1',\dd_2')\\ 
  &\ge \Big(1\!+s_i\frac{\|x_i-x\|}{\dd_1}\Big)I_V(x,x^*,r,\dd_1,\dd_2)-\!\frac1{\dd_1} \big(\|x^*\|\|x_i-x\|+\dd_2\|x^*-x^*_i\|\big)-\gg\frac{\dd_2}{\dd'_1}>\!-\infty,
 \end{aligned}
\end{eqnarray} 
and, in particular $I_V(x_i,x^*_i,q,\dd_1',\dd_2')>-\infty$, holds for every $i\in N_1$. 

Now, put 
\begin{eqnarray}\label{84}
N_2:=\big\{i\in N_1:\ \dd_1'<\|h+x-x_i\|<\dd_2'\ {\rm and}\ \langle x^*_i,x-x_i\rangle 
+\langle x^*_i-x^*,h\rangle > -\|h\|\gamma\big\};
\end{eqnarray}
this is still a co-finite set in $\N$. 
Using (\ref{16}), for every $i\in N_2$ we can estimate, 
\begin{eqnarray}\label{18}
 \begin{aligned}
  &\frac1{\|h\|}\big(f(x+h)-r-\langle x^*,h\rangle\big) \\ 
  =&\frac{\|x-x_i+h\|}{\|h\|}\cdot\frac1{\|x-x_i+h\|}\big(f(x_i+(x-x_i+h))-q-\langle x^*_i,x-x_i+h\rangle\big)\\  
  & +\frac1{\|h\|}\big(\la x^*_i,x-x_i\ra +\la x^*_i-x^*,h\ra\big) + \frac{q-r}{\|h\|}\\
  >& \frac{\|x-x_i+h\|}{\|h\|}I(x_i,x^*_i,q,\delta'_1,\delta'_2) -\gg-\gg \\  
  \ge &\frac{\|x-x_i+h\|}{\|h\|}\bigg[\frac1{\|h(x_i,x^*_i,q,\dd_1',\dd_2',\gamma)\|}
  \big(f\big(x_i+h(\cdots)\big)-q-\big\la x^*_i,h(\cdots)\big\ra\big)-\gamma\bigg]-2\gg\\ 
  \ge &\frac{\|x-x_i+h\|}{\|h\|}\big[I_V\big(x_i,x^*_i,q,\dd_1',\dd_2',\big)-\gamma\big]-2\gg; 
 \end{aligned}
\end{eqnarray}
here $\cdots$ meant the cortege $x_i, x^*_i, q,\delta'_1,\delta'_2,\gamma$. Now, plugging here (\ref{36}), and then letting $N_2\ni i\rightarrow\infty$, we get that
$$
\frac1{\|h\|}\big(f(x+h)-r-\la x^*,h\ra\big) \ge I_V(x,x^*,r,\dd_1,\dd_2) - 3\gamma -\gg\frac{\dd_2}{\dd'_1}\,,
$$
Finally, realizing that $\gg\in\Q_+$ could be arbitrarily small, we get that $\frac1{\|h\|}\big(f(x+h)-r-\la x^*,h\ra\big) \ge I_V(x,x^*,r,\dd_1,\dd_2)$.
This, of course, implies that $I(x,x^*,r,\dd_1,\dd_2)\ge I_V(x,x^*,r,\dd_1,\dd_2)$.
\smallskip

The proof of $\sigma$-completeness of $\RR$ is very similar to (but a bit different from) the proof of cofinality. 
Let $V_1,\times Y_1,\ V_2\times Y_2,\ \ldots$ be an increasing sequence of elements in our $\RR$.
We have to verify that $\overline{V_1\times Y_1 \cup V_2\times Y_2\cup \cdots}\,$ also belongs to $\RR$.
Clearly, this set is of form $V\times Y$. As $\AA$ is $\sigma$-complete, $V\times Y\in\AA$.
It remains to verify (\ref{17}). 
So, fix any cortege $x,x^*,r,\dd_1$, and $\dd_2$ as there. Consider any $h\in X$ such that $\dd_1<\|h\|<\dd_2$.
We have to show that  $\frac1{\|h\|}(f(x+h)-r-\langle x^*,h\rangle) \ge I_V(x,x^*,r,\dd_1,\dd_2)$.
This inequality is trivially satisfied if $I_V(x,x^*,r,\dd_1,\dd_2)=-\infty$. Further assume that this is not so.
Pick some $\dd_1',\dd_2'\in\Q$ such that $\dd_1<\dd_1'<\|h\|<\dd_2'<\dd_2$. 
It is easy to check that $V=\overline{V_1\cup V_2\cup\cdots}\, $ and $Y=\overline{Y_1\cup Y_2\cup\cdots}\,$.
Find $x_1\in V_1,\ x_2\in V_2,\ \ldots$ and $x^*_1\in Y_1,\ x^*_2\in Y_2,\ \ldots $ such that $\|x_i-x\|\longrightarrow0$
and $\|x^*_i-x^*\|\longrightarrow0$ as $i\rightarrow\infty$. Consider any fixed $\gg\in\Q_+$. 
Pick $q\in\Q$ such that $|q-r|<\gamma\|h\|$. Denote 
$N_1:= \big\{i\in \N:\ \|x_i-x\|<\min\{\dd_1'-\dd_1,\dd_2-\dd_2'\}\big\}$;
this is a co-finite set in $\N$. Now, take any $k\in V$, with $\dd_1'<\|k\|<\dd_2'$. 
For $i\in N_1$ we have $\dd_1<\|x_i-x+k\|<\dd_2$ and then we can estimate 
(This chain is exactly as (\ref{81}).)
\begin{eqnarray*}
&&\frac1{\|k\|}\big(f(x_i+k)-q-\la x^*_i,k\ra\big)\\
  &=&\frac{\|k+x_i-x\|}{\|k\|}\cdot\frac1{\|k+x_i-x\|}\big(f(x+(x_i-x+k))-r-\la x^*,x_i-x+k\ra\big)\\
&&+\frac1{\|k\|}\big(\la x^*,x_i-x+k\ra - \la x^*_i,k\ra\big) + \frac{r-q}{\|k\|}\\
&\ge& \Big(1+s_i\frac{\|x_i-x\|}{\|k\|}\Big)I_V(x,x^*,r,\dd_1,\dd_2)-	\frac1{\dd_1}\big(\|x^*\|\|x_i-x\|+\dd_2\|x^*-x^*_i\|\big)
-\gg\frac{\dd_2}{\dd'_1}
\end{eqnarray*}
where $s_i=1$ if $I_V(x,x^*,r,\dd_1,\dd_2)\le0$ and $s_i=-1$ otherwise.
It then follows that (This is exactly as (\ref{36}).)
\begin{eqnarray}\label{336}
 \begin{aligned}
&I_V(x_i,x^*_i,q,\dd_1',\dd_2')\\ 
\quad\quad\ge& \Big(1+s_i\frac{\|x_i-x\|}{\dd_1}\Big)I_V(x,x^*,r,\dd_1,\dd_2)-
\frac1{\dd_1}\big(\|x^*\|\|x_i-x\|+\dd_2\|x^*-x^*_i\|\big)-\gg\frac{\dd_2}{\dd'_1}>-\infty,
 \end{aligned}
\end{eqnarray} 
and, in particular $I_V(x_i,x^*_i,q,\dd_1',\dd_2')>-\infty$, holds for every $i\in N_1$. 

Now, put (This $N_2$ is defined exactly as in (\ref{84}).)
$$
N_2:=\big\{i\in N_1:\ \dd_1'<\|h+x-x_i\|<\dd_2'\  {\rm and}\ \langle x^*_i,x-x_i\rangle 
+ \langle x^*_i-x^*,h\rangle>-\gamma\|h\|\big\};
$$
this is still a co-finite set in $\N$. 
Using (\ref{336}), for every $i\in N_2$ we can estimate (The following chain is a bit different from (\ref{18}).) 
\begin{eqnarray*}
&&\frac1{\|h\|}\big(f(x+h)-r-\langle x^*,h\rangle\big) \\ \nonumber
&=&\frac{\|x-x_i+h\|}{\|h\|}\cdot\frac1{\|x-x_i+h\|}\big(f(x_i+(x-x_i+h))-q-\langle x^*_i,x-x_i+h\rangle\big)\\ \nonumber 
&& +\frac1{\|h\|}\big(\la x^*_i,x-x_i\ra +\la x^*_i-x^*,h\ra\big) + \frac{q-r}{\|h\|}\\ \nonumber
&>& \frac{\|x-x_i+h\|}{\|h\|} I(x_i,x^*_i,q,\dd'_1,\dd'_2)-\gg-\gg \\ \nonumber
&=& \frac{\|x-x_i+h\|}{\|h\|} I_{V_i}(x_i,x^*_i,q,\dd'_1,\dd'_2)-2\gg\ \ \
\hbox{\rm (as $(x_i,x^*_i)\in V_i\times Y_i\in\RR$ and (\ref{17}) holds)} \\ \nonumber
&\ge& \frac{\|x-x_i+h\|}{\|h\|} I_V(x_i,x^*_i,q,\dd'_1,\dd'_2)-2\gg.
\end{eqnarray*}
Now, plugging here (\ref{336}), and then letting $N_2\ni i\rightarrow\infty$, we get that
$$
\frac1{\|h\|}\big(f(x+h)-r-\la x^*,h\ra\big) \ge I_V(x,x^*,r,\dd_1,\dd_2) - 2\gamma -\gg\frac{\dd_2}{\dd'_1}\,,
$$
Finally, realizing that $\gg\in\Q_+$ could be arbitrarily small, we get that $\frac1{\|h\|}\big(f(x+h)-r-\la x^*,h\ra\big) \ge I_V(x,x^*,r,\dd_1,\dd_2)$.
This, of course, implies that $I(x,x^*,r,\dd_1,\dd_2)\ge I_V(x,x^*,r,\dd_1,\dd_2)$.
We proved that $\RR$ is $\sigma$-complete, and therefore $\RR$ is a rich rectangle family
in $X\times X^*$.
\smallskip

That $Y_1\subset Y_2$ whenever $V_1\times Y_1,\ V_2\times Y_2\in\RR$ and $V_1\subset V_2$, 
follows immediately from the same property shared by $\AA$. 
\smallskip 

It remains to prove that our $\RR$ ``works''. So, pick any $V\times Y\in\RR$. We know 
from Theorem~\ref{dno} that $Y\ni x^*\longmapsto x^*|_V\in V^*$ is (an isometry) onto. Assume there is
$(v,v^*)\in V\times V^*$ such that $v^*\in\partial_F(f|_V)(v)$. Find (a unique) $x^*\in Y$ such that
$x^*|_V=v^*$. We shall show that $x^*\in\partial_F f(v)$. So, fix any $\ee>0$. Find $\delta>0$ such that
$f(v+k)-f(v)-\la v^*,k\ra >-\ee \|k\|$ whenever $k\in V$ and $0<\|k\|<\delta$; then $I_V(v,v^*,f(v),\dd_1,\dd)\ge-\ee$.
Now, let $h\in X$ be any vector such that $0<\|h\|<\delta$. Pick $\dd_1\in(0,\|h\|)$.
Then we have
\begin{eqnarray*}
\frac1{\|h\|}\big(f(v+h)-f(v)-\la x^*,h\ra\big) \ge I(v,x^*,f(v),\dd_1,\dd) =I_V(v,v^*,f(v),\dd_1,\dd)\ge-\ee 
\end{eqnarray*} 
by (\ref{17}). We proved that $x^*$ belongs to $\partial f(v)$, and so $v^*$ belongs to $\big(\partial_Ff(v)\cap Y\big)|_V$.
\end{proof}

\begin{corollary}\label{123}
Let $(X,\nn)$ be a (rather non-separable) Asplund space and let $f:X\longrightarrow(-\infty,+\infty]$ be any proper function. Then there exists a rich family $\RR\subset\SS(X)$ such that 
for every $V\in \RR$ and for every $v\in V$ we have: 
\item{(i)} $\partial_F f(v)\neq\emptyset$ if (and only if) $\partial_F(f|_V)(v)\neq\emptyset$ (see \cite{f3}).
\item{(ii)} For $t\ge0$ we have $\partial_F f(v)\setminus tB_{X^*}\neq\emptyset$ if (and only if) $\partial_F(f|_V)(v)\setminus tB_{V^*}\neq\emptyset$ (see \cite{fm}).
\item{(iii)} $f$ is Fr\'echet differentiable at $v$ if (and only if) $f|_V$ is Fr\'echet differentiable at $v$; 
and in this case $\|f'(v)\|=\big\|(f|_V)'(v)\big\|$ (see \cite{pr,z}).
\end{corollary}

\begin{proof}
Let $\RR_1$ and $\RR_2$ be rich rectangle-families found in Theorem~\ref{nova} for the functions $f$ and $-f$,
respectively. Let $\RR$ be the ``projection of $\RR_1\cap\RR_2$ on the first coordinate'',
that is, put 
$$
\RR:=\big\{V\in\SS(X):\ V\times Y\in\RR_1\cap\RR_2\ \ \rm for\  some\ Y\in\SS(X^*)\big\}.
$$
It is easy check that $\RR$ is rich. It works. Indeed, (i) and (ii) follow immediately from Theorem~\ref{nova}. As regards (iii), take any $V\in\RR$ and any $v\in V$.
Find $Y\in\SS(X^*)$ so that $V\times Y$ is in $\RR_1\cap\RR_2$. Then
(i) and (ii) immediatley follow from Theorem~\ref{nova}.
Further, assume that $f|_V$ is Fr\'echet differentiable at $v$ and put $v^*:=(f|_V)'(v)$. This implies that $v^*\in\partial_F (f|_V)(v)$
and $-v^*\in\partial_F ((-f)|_V)(v)$. Find the unique $x^*\in Y$ such that $x^*|_V=v^*$ and $\|x^*\|=\|v^*\|$;
then $(-x^*)|_V=-v^*$. Now, by Theorem~\ref{nova},
$x^*\in\partial_Ff(v)$ and $-x^*\in\partial_F(-f)(v)$. It then follows that 
$f$ is Fr\'echet differentiable at $v$, with $f'(v)=x^*$
and $\|f'(v)\|= \|x^*\|=\|v^*\|=\|(f|_V)'(v)\|$.
\end{proof}

\begin{corollary}\label{fuzzy} {\rm (\cite{fm})}
Let $(X,\|\cdot\|)$ be an Asplund space, let $f:X\longrightarrow(-\infty,+\infty]$ be a lower semicontinuous function,
and $g:X\longrightarrow(-\infty,+\infty]$ be a function uniformly continuous in a vicinity of a certain $\overline x\in X$.
Then: 
\item{(i)} The set $\{x\in X:\ \partial_Ff(x)\neq\emptyset\}$ is dense in the domain of $f$.

\item{(ii)} If $x^*\in\partial_F(f+g)(\overline x)$, then for every $\ee>0$
there are $x_1, x_2\in X,\ x_1^*\in\partial_Ff(x_1)$, and $x_2^*\in\partial_Fg(x_2)$ such that
$\| x_1-\overline x\|<\ee$, $\| x_2-\overline x\|<\ee$,  and $\|x_1^*+x_2^*-x^*\|<\ee$.
\end{corollary} 
  	
\begin{proof} Assume first that $X$ is separable. 
Find an equivalent Fr\'echet smooth norm $|\cdot|$, see e.g. \cite[pages 48, 43]{dgz} arbitrraily close to $\|\cdot\|$. 
Then proceed as in \cite{AI83} and \cite{bm06}, using Borwein-Preiss or Deville-Godefroy-Zizler smooth variational principles \cite[Section 4]{ph}.

Second, assume that $X$ is non-separable. As regards (i), combine
the just proved separable statement with Corollary~\ref{123} (i). 
To prove (ii), assume that $x^*\in\partial_F(f+g)(\overline x)$ and let $\ee>0$
be given. By Theorem~\ref{nova}, find rich families $\RR_1,\ \RR_2$ corresponding to $f, g$,
respectively, and put $\RR:=\RR_1\cap\RR_2$. Find $V\times Y\in\RR$ so that it contains $(\overline x,x^*)$. 
Using the validity of the separable statement, find $x_1,x_2\in V,\ v^*_1\in\partial_F(f|_V)(x_1)$, and
$v^*_2\in\partial_F(g|_V)(x_2)$ such that $\| x_1-\overline x\|<\ee,\ \|x_2-\overline x\|<\ee$, and $\|v_1^*+v_2^*-x^*|_V\|<\ee$.
Now, the conclusion of Theorem~\ref{nova} provides unique $x^*_1\in\partial_Ff(x_1)\cap Y$ and $x^*_2\in\partial_Fg(x_2)\cap Y$
such that $x^*_i|_V=v^*_i,\ i=1,2$. Hence, using the isometric property of the restriction mapping $Y\ni\xi\longmapsto \xi|_V$,
we conclude that $\|x_1^*+x_2^*-x^*\|=\|v_1^*+v_2^*-x^*|_V\|<\ee$.
\end{proof}

Let $(X,\nn)$ be a Banach space, let $\Omega \subset X$, and let $\overline x\in\Omega$.
The {\it Fr\'echet normal cone} $N_F(\overline x,\Omega)$ of $\Omega$ at $\overline x$
is defined as the Fr\'echet subdifferential of the indicator function $\iota_\Omega$ at $\overline x$;
note that $N_F(\overline x,\Omega)$ always contains $0$. By an \it extremal system \rm in $X$ we
understand any triple $(\Omega_1,\Omega_2,\overline x)$ such that $\Omega_1,\Omega_2$ are subsets
of $X$, the point $\overline x$ lies in $\Omega_1\cap\Omega_2$, and there are $\ee>0$ and sequences
$(a^1_n),\ (a^2_n)$ in $X$ satisfying that $(a^1_n+\Omega_1)\cap(a^2_n+\Omega_2)\cap(\overline x+\ee B_X)=\emptyset$
for every $n\in\N$.

\begin{corollary}\label{ep}{\rm (\cite{fm})}
Let $(X,\|\cdot\|)$ be an Asplund space and let $(\Omega_1,\Omega_2,\overline x)$ be an extremal system of closed sets in $X$.
Then: 
\item{(i)} The set $\big\{x\in X:\ N_F(\overline x,\Omega_1)\neq\{0\}\big\}$ is dense in the boundary of $\Omega_1$.

\item{(ii)} The Fr\'echet extremal principle for the triple $(\Omega_1,\Omega_2,\overline x)$ holds,
that is, for every $\ee>0$ there are $x_1, x_2\in X$ such that $\| x_1-\overline x\|<\ee$, $\| x_2-\overline x\|<\ee$
and there are $x_i^*\in N_F(x_i,\Omega_i)+\ee B_{X^*},\ i=1,2$, such that $\|x^*_1\|+\|x^*_2\|=1$, and $x^*_1+x^*_2=0$.
\end{corollary}

\noindent The proof is very similar to that of Corollary~\ref{fuzzy}, once we have at hand the ``separable'' statements; \cite{fm}.  

We finish this section by deriving easily a strengthening of the main result of the paper \cite{fi2} from
Theorems~\ref{dnoo} and \ref{nova} in the framework of Asplund spaces.

\begin{theorem}\label{princip}
Let $k\in \N$, let $X$ be a non-separable {\tt Asplund} space, let $Z_1,\ldots, Z_k$ be Banach spaces, 
let $z^*_1\in Z^*_1,\ \ldots, z^*_k\in Z^*_k$, let $T_i: Z_i\to X,\ i=1,\ldots,k$,  
be bounded linear operators, and let $f:X\longrightarrow(-\infty,+\infty]$ be a proper function. 
Then there exists a rich block-family $\RR\subset\SS_{{\oo}^{\!\!\!\!\!\oo}}(Z_1\times\cdots\times Z_k
\times X)$ such that, for every $U_1\times\cdots\times U_k\times V\in\RR$ we have 
$T_1(U_1)\subset V,\ \ldots,\ T_k(U_k)\subset V$ and there is $Y\in\SS(X^*)$ such that:

\item{(i)} The assignment $Y\ni x^*\longmapsto x^*|_V\in V^*$ is an isometry {\tt onto} $V^*$;

\item{(ii)} $\big(\partial_Ff(v)\cap Y\big)|_V=\big(\partial_Ff(v)\big)|_V=\partial_F(f|_V)(v)$
for every $v\in V$; and

\item{(iii)} $\big\|T^*_ix^*-z^*_i\big\|=\big\|\big(T_i|_{U_i}\big)^*(x^*|_V)-(z^*_i|_{U_i})\big\|$ for every $x^*\in Y$ and $i=1,\ldots,k$
\end{theorem} 

\begin{proof}
Putting together Theorem~\ref{dnoo} and Remark~\ref{2.6}, we find a rich block-family $\AA_{T_1,\ldots,T_k}$ in
$Z_1\times\cdots\times Z_k\times X\times X^*$ with similar properties  
as the family $\AA_T$ in Theorem~\ref{dnoo} has. Let $\RR'$ be the rich family in $X\times X^*$ found in Theorem~\ref{nova}. Define 
$$
\RR:=\big\{U_1\times\cdots\times U_k\times V:\ U_1\times\cdots\times U_k\times V\times Y\in\AA_{T_1,\ldots,T_k}
\ {\rm and}\ V\times Y\in\RR'
\ \hbox{for some}\ Y\in\SS(X^*)\big\}.
$$
Clearly, $\RR$ is cofinal. And from the ``monotonicity'' property of $\AA_{T_1,\ldots,T_k}$ we easily get that $\RR$ is $\sigma$-complete. That (i) and (iii) are true follows directly from
Theorem~\ref{dnoo}. (ii) comes immediately from Theorem~\ref{nova}.
\end{proof}

\begin{remark} {\em From the theorem above we immediately get a strengthening of
\cite[Corollary 4.5]{fi2} and afterwards \cite[Theorem 4.6]{fi2}.}
\end{remark}

\section{An approach through suitable models from logic}

Here we provide an alternative approach to the proofs presented in the previous sections. The method we used so far is called ``the method of rich families''. An alternative approach we apply below is called ``the method of suitable models''. The reason why we present proofs using this method  is that it is shorter and less technical. Hence, we hope that this presentation will encourage a reader, not familiar with set theory or logic, to get in touch with this technology as well. A relationship between both methods is discussed in Section \ref{section5}. Now 
we just mention that, in the setting of Asplund spaces, both methods are in some sense equivalent;
see Theorem \ref{t:richIffModel} and the discussion below it.

We recall some basic definitions and facts concerning the method of suitable models. A brief description of this method can be found in \cite{crz}; for a more detailed description see \cite{c}. Let $N$ be a fixed set and $\phi$ a formula in the basic language of the set theory. By the \emph{relativization of $\phi$ to $N$} we understand the formula $\phi^N$ which is obtained from $\phi$ by replacing each chain of the form ``$\forall x$'' by ``$\forall x \in N$'' and each chain of the form ``$\exists x$'' by ``$\exists x \in N$''. Let $\phi(x_1,\ldots,x_n)$ be a formula with all free variables shown, i.e., a formula whose free variables are exactly $x_1,\ldots,x_n$. We say that \emph{$\phi$ is absolute for $N$} if
\[
\forall a_1, \ldots, a_n \in N\colon \bigl(\phi^N(a_1,\ldots,a_n) \leftrightarrow \phi(a_1,\ldots,a_n)\bigr).
\]
The method is based mainly on the following theorem (a proof can be found in \cite[Chapter IV, Theorem 7.8]{k}).
The cardinality of a set $A$ is denoted by $|A|$.

\begin{theorem}\label{T:countable-model}
Let $\phi_1, \ldots, \phi_n$ be any formulas and $X$ be any set. Then there exists a set $M \supset X$ such that
$\phi_1, \ldots, \phi_n \text{ are absolute for } M$ and $|M| = \max\,(\aleph_0,|X|)$.
\end{theorem}

The following notation is useful.

\begin{definition}
\rm Let $\Phi$ be a finite list of formulas and $X$ be any countable set.
Let $M \supset X$ be a {\tt countable} set such that each $\phi$ from $\Phi$ is absolute for $M$.
Then we say that $M$ is a \emph{suitable model for $\Phi$ containing $X$}.
This is denoted by $M \prec (\Phi; X)$.
\end{definition}

\noindent Let us emphasize that a suitable model in our terminology is always countable.

The fact that a certain formula is absolute for $M$ will always be used in order to satisfy
the assumption of the following lemma; see \cite[Lemma 2.3]{crz}.

\begin{lemma}\label{L:unique-M}
Let $\phi(y,x_1,\ldots,x_n)$ be a formula whose all free variables are shown, and let
$M$ be a fixed set such that both formulas ``$\phi$'' and ``$\,\exists y \colon \phi(y,x_1,\ldots,x_n)\!$'' are absolute for
$M$. Assume that there exist $a_1,\ldots,a_n \in M$ and $u$ such that $\phi(u,a_1,\ldots,a_n)$ holds.
Then there exists a set $a \in M$ such that $\phi(a,a_1,\ldots,a_n)$ holds.
\end{lemma}

Instead of the basic language of the set theory, we use extended language of the set theory as we are used to (for a detailed discussion see e.g. \cite[Section 2]{crz}).
We shall also use the following convention.

\begin{convention}\hfil\\
$\bullet$ If $(X, +, \cdot, \| \cdot \|)$ is a normed linear space and $M$ is a suitable model (for some formulas containing some set), then by writing $X\in M$ (or by writing $\{X\}\subset M$) we mean that 
$X,\;+,\;\cdot,\;\text{\mbox{$\| \cdot \|$}} \in M$.\\
$\bullet$ If $X$ is a topological space and $M$ is a suitable model, then we denote by $X_M$ the set $\overline{X\cap M}$; clearly, the set $X_M$ is separable.
\end{convention}

Finally, we recall several further results about suitable models
(all the proofs are based on Lemma \ref{L:unique-M} and they can be found in \cite[Sections 2 and 3]{c}).

\begin{lemma}\label{l:basics-in-M}
There are a finite list of formulas $\Phi$ and a countable set $C$ such that any $M \prec (\Phi; C)$ satisfies the following:
\item{(i)}  Let $f$ be a function such that $f\in M$. Then, for every $x \in \operatorname{Dom} f \cap M$ we have $f(x) \in M$.

\item{(ii)} If $A,B\in M$, then $A\cap B\in M$, $B\setminus A\in M$ and $A\cup B\in M$.

\item{(iii)} If $S \in M$ is a countable set, then $S \subset M$.

\item{(iv)} If $X$ is a normed linear space and $X\in M$, then $X_M := \overline{X\cap M}$
is a (closed separable) linear subspace.
\end{lemma}

In the sequel we will need the following simple observation.

\begin{lemma}\label{l:dualInM}There are a finite list of formulas $\Phi$ and a countable set $C$ such that any $M \prec (\phi; C)$ satisfies the following:\\[5pt]
If $X$ is a Banach space with $X\in M$, then $X^*\in M$ and $\overline{X^*\cap M}$ is a subspace of $X^*$.
\end{lemma}

\begin{proof}Let us fix a finite list of formulas $\Psi$ and a countable set $C$  from the statement of Lemma \ref{l:basics-in-M}. Append to $\Psi$ the formula marked with $(*)$ below and all its subformulas.  Denote such an enhanced list of formulas by $\Phi$ and fix any $M \prec (\Phi; C)$ with $X\in M$. By Lemma \ref{L:unique-M}, using the absoluteness of the following formula (and its subformula)
\[\tag{$*$}
\exists X^* \colon\ (X^* \text{ is the dual space to }X),
\]
we get that $X^*\in M$. Now, by Lemma \ref{l:basics-in-M}, $\overline{X^*\cap M}$ is a subspace of $X^*$.
\end{proof}

Let us start with an analogue to Theorem \ref{dnoo}. Note that here we  do not assume that $X$ is Asplund.

\begin{theorem}\label{dnoo-modely} There are a finite list of formulas $\Phi$ and a countable set $C$ such that any $M \prec (\Phi; C)$ satisfies the following:\\[5pt]
Let $Z$ and $X$ be Banach spaces, let $T:Z\to X$ be a bounded linear operator, and let $z^*\in Z^*$ be given. Assume that $\{Z, X, T, z^*\}\subset M$. Then $T(Z_M)\subset X_M$ and, for every $x^*\in \overline{X^*\cap M}$, we have $\|T^*x^* - z^*\| = \|(T|_{Z_M})^*(x^*|_{X_M}) - (z^*|_{Z_M})\|$.
In particular, if $X$ is a Banach space with $X\in M$ then, for every $x^*\in \overline{X^*\cap M}$, we have $\|x^*\| = \|x^*|_{X_M}\|$.
\end{theorem}

\begin{proof}Let us fix a finite list of formulas $\Psi$ and a countable set $C$  from the statement of Lemma \ref{l:basics-in-M}. We may assume that $C\supset \Q$ for otherwise we replace it by $C\cup \Q$. 
Append to $\Psi$ the two formulas marked with $(*)$ below and all subformulas of them.  Denote such an
extended list of formulas by $\Phi$ and fix any $M \prec (\Phi; C)$ with $\{Z, X, T, z^*\}\subset M$. By Lemma \ref{l:basics-in-M}, $Z_M$ (resp. $X_M$) is a subspace of $Z$ (resp. $X$) and we have $T(Z\cap M)\subset X\cap M$; hence, $T(Z_M)\subset X_M$.

Fix $x^*\in \overline{X^*\cap M}$. It is obvious that $\|T^*x^* - z^*\| \geq \|(T|_{Z_M})^*(x^*|_{X_M}) - (z^*|_{Z_M})\|$. Let us fix $n\in\N$ and find $x_0^*\in X^*\cap M$ such that $\|x^*-x_0^*\|<\tfrac{1}{n\|T\|}$. Further, find $q\in\Q$ with $|\|T^*x^* - z^*\| - q| \leq \tfrac 1n$ and $z\in B_Z$ with $\|T^*x^* - z^*\|\leq |T^*x^*(z) - z^*(z)| + \tfrac 1n$. Then the following formula is true:
\[\tag{$*$}
\exists z\in Z :\  \big(\|z\|\leq 1\; \wedge\; q\leq |T^*x_0^*(z) - z^*(z)| + \tfrac 3n\big).
\]
Using Lemma~\ref{L:unique-M} and absoluteness of the preceding formula and its subformula for $M$, we find the corresponding $z \in Z\cap M$. Hence, $z\in B_{Z_M}$ and
\[
\|T^*x^* - z^*\|\leq q + \tfrac 1n\leq |T^*x_0^*(z) - z^*(z)| + \tfrac 4n\leq |T^*x^*(z) - z^*(z)| + \tfrac 5n\leq \|(T|_{Z_M})^*(x^*|_{X_M}) - (z^*|_{Z_M})\| + \tfrac 5n.
\]
As $n\in\N$ was arbitrary, we get $\|T^*x^* - z^*\|\leq \|(T|_{Z_M})^*(x^*|_{X_M}) - (z^*|_{Z_M})\|$.

The ``in particular'' part easily follows. Indeed, if $M$ is as above, we just observe that, by Lemma \ref{L:unique-M}, using the absoluteness of the following formula (and its subformulas)
\[\tag{$*$}
\exists \operatorname{id_X}:\;(\operatorname{id_X}:X\to X\text{ is a mapping defined by }
\operatorname{id}_X(x): = x,\ x\in X),
\]
we have that $\operatorname{id_X}\in M$ and we may use the result with $T:= \operatorname{id_X}$.
\end{proof}

The following statement is an analogue of the implication (ii)$\implies$(iii)
in Theorem \ref{dno}. This result was proved in \cite[Theorem 5.4]{crz}. We present here a proof using the notion of Asplund generator.

\begin{theorem}\label{dno-modely}
There are a finite list of formulas $\Phi$ and a countable set $C$ such that any $M \prec (\Phi; C)$ satisfies the following:\\[5pt]
Let $(X,\nn)$ be a (rather non-separable) Asplund space with $X\in M$. Then $\overline{X^*\cap M}\ni x^*\longmapsto x^*|_{X_M}\in (X_M)^*$ is an isometry {\tt onto} $(X_M)^*$.
\end{theorem}

\begin{proof}Let us fix finite lists of formulas $\Psi$, $\widetilde{\Psi}$ and countable sets $C$, $\widetilde{C}$ from the statements of Lemma \ref{l:basics-in-M} and Theorem \ref{dnoo-modely}, respectively. 
Let $\Phi$ be the list of formulas consisting of formulas from $\Psi$, $\widetilde{\Psi}$, 
the formula marked with $(*)$ below and all its subformulas. Fix any $M \prec (\Phi; C\cup\widetilde C)$,
with $M\ni X$. By Lemma \ref{l:basics-in-M}, $X_M$ is a (closed separable) subspace of $X$. By Theorem \ref{dno} (i)$\implies$(ii), \[\tag{$*$}
\exists G \colon (G\text{ is an Asplund generator in $X$}).
\]
Using the absoluteness of this formula (and its subformula) for $M$, by Lemma \ref{L:unique-M}, we have an Asplund generator $G$ in $X$ with $G\in M$. For every $x\in X\cap M$, by Lemma \ref{l:basics-in-M} (i) (iii), we have $G(x)\subset X^*\cap M$; thus, $G(X\cap M)\subset X^*\cap M$. Hence,
$$(X_M)^* \stackrel{\ref{generator} (a)}{\subset} \overline{G(X\cap M)|_{X_M}}\subset \overline{(X^*\cap M)|_{X_M}} \subset (X_M)^*$$
and we get that $\overline{(X^*\cap M)|_{X_M}} = (X_M)^*$. It is always true that $\overline{(X^*\cap M)|_{X_M}}\supset \overline{X^*\cap M}|_{X_M}$ and the other inclusion follows using the fact that, by Theorem \ref{dnoo-modely}, $\overline{X^*\cap M}\ni x^*\longmapsto x^*|_{X_M}$ is an isometry.
\end{proof}

Finally, the following statement is an analogue of Theorem \ref{nova}. (Note that here we  do not assume that $X$ is Asplund.)

\begin{theorem}\label{nova-modely} There are a finite list of formulas $\Phi$ and a countable set $C$ such that any $M \prec (\Phi; C)$ satisfies the following:\\[5pt]
Let $X$ be a Banach spaces and $f:X\longrightarrow(-\infty,\infty]$ be a proper function with $\{X, f\}\subset M$. Then for every $x\in X_M$, with $f(x) < +\infty$, and every $x^*\in \overline{X^*\cap M}$, we have
$$
x^*\in \partial_F f(x)\; \Longleftrightarrow\; x^*|_{X_M}\in \partial_F (f|_{X_M})(x).
$$
\end{theorem}

\begin{proof}Let us fix a finite list of formulas $\Psi$ and a countable sets $C$  from the statement of Lemma \ref{l:basics-in-M}. We may assume that $C\supset \Q$ for otherwise we replace it by $C\cup \Q$. 
Enhance the list $\Psi$ by the formula marked with $(*)$ below and by all subformulas of it.  Denote such an
augmented list by $\Phi$ and fix any $M \prec (\Phi; C)$ with $\{X, f\}\subset M$. By Lemma \ref{l:basics-in-M}, $X_M$ is a subspace of $X$.

Fix any $x\in X_M$ with $f(x) < \infty$ and $x^*\in \overline{X^*\cap M}$. Obviously, $x^*|_{X_M}\in \partial_F (f|_{X_M})(x)$ whenever $x^*\in \partial_F f(x)$. In order to prove the other implication, let us assume that $x^*\notin \partial_F f(x)$. Then there exists $\varepsilon\in\Q_+$ such that
\begin{equation}\label{eq:frechet-subdif-eps}
\forall\delta\in\Q_+\; \exists h\in B(0,\delta)\setminus\{0\}:\; f(x+h) - f(x) - x^*(h)\leq -\varepsilon\|h\|.
\end{equation}
In order to see that $x^*|_{X_M}\notin \partial_F (f|_{X_M})(x)$, we will verify that
\begin{equation}\label{eq:frechet-subdif-eps3}
\forall\delta\in\Q_+\; \exists h''\in B(0,\delta)\cap X_M\setminus\{0\}:\; f(x+h'') - f(x) - x^*(h'') < -\tfrac{\varepsilon}{3}\|h''\|.
\end{equation}
Fix $\delta\in\Q_+$. By \eqref{eq:frechet-subdif-eps}, there exists $0\neq h\in B(0,\tfrac{\delta}{3})$ such that
\begin{equation}\label{eq:frechet-subdif-f}
f(x+h) - f(x) - x^*(h)\leq -\varepsilon\|h\|.
\end{equation}
Find $q\in(0,\|h\|)\cap \Q$. Now, fix $\eta\in\Q_+$ such that $\eta < \min\{\tfrac{\delta}{6},\tfrac{q}{2},\tfrac{q\varepsilon}{3(1+\|x^*\|+\delta/2+4\varepsilon/3)}\}$. Take $q'\in\Q$ with $|f(x) - q'| \leq \eta$, $x_0\in B(x,\eta)\cap M$ and $x_0^*\in B(x^*,\eta)\cap M$. Put $h' = h + (x - x_0)$. Then
\[
q - \eta < \|h\| - \eta \leq \|h'\|\leq \|h\| + \eta\leq \tfrac{\delta}{3} + \eta < \tfrac{\delta}{2},
\]
and
\begin{align*}
f(x_0 + h') - q' - x_0^*(h')  & \leq f(x + h) - f(x) - x^*(h) + \eta + \eta\|h\| + \eta\|x_0^*\| \\
& \!\! \stackrel{\eqref{eq:frechet-subdif-f}}{\leq} -\varepsilon\|h\| + \eta(1+\tfrac{\delta}{3} + \|x^*\| + \eta) \\
& < -\varepsilon\|h'\| + \eta(1 + \tfrac{\delta}{3} + \|x^*\| +\tfrac{\delta}{6}+ \varepsilon) \\
& = -\varepsilon\|h'\| + \eta(1 + \tfrac{\delta}{2} + \|x^*\| + \tfrac{4}{3}\varepsilon) - \tfrac{\varepsilon}{3}\eta \\
& < -\varepsilon\|h'\| + \tfrac{\varepsilon}{3}(q -  \eta) < -\varepsilon\|h'\| + \tfrac{\varepsilon}{3}\|h'\| = -\tfrac{2}{3}\varepsilon\|h'\|.
\end{align*}
Hence, the following formula is true:
\[\tag{$*$}
\exists h'\in X : (q - \eta < \|h'\| < \delta/2\quad \wedge\quad f(x_0 + h') - q' - x_0^*(h') < -\tfrac{2}{3}\varepsilon\|h'\|).
\]
Using Lemma~\ref{L:unique-M} and absoluteness of the preceding formula and its subformula for $M$, we find $h' \in X\cap M$ satisfying the formula above. Let us notice, that this $h'$ need not satisfy $h' = h + (x - x_0)$. Put $h'' = h' + (x_0 - x)$. Then $h''\in B(0,\delta)\cap X_M\setminus\{0\}$ and
\begin{align*}
f(x + h'') - f(x) - x^*(h'') & \leq f(x_0 + h') - q' - x_0^*(h') + \eta + \eta\|h'\| + \eta\|x^*\| \\
& < -\tfrac{2}{3}\|h'\|\varepsilon + \eta(1 + \tfrac{\delta}{2} + \|x^*\|) \\
& < -\tfrac{2}{3}\|h''\|\varepsilon + \eta(1 + \tfrac{\delta}{2}+ \|x^*\| + \tfrac{2}{3}\varepsilon)\\
& = -\tfrac{2}{3}\|h''\|\varepsilon + \eta(1 + \tfrac{\delta}{2} + \|x^*\| + \tfrac{4}{3}\varepsilon) - \tfrac{2}{3}\eta\varepsilon\\
& < -\tfrac{2}{3}\|h''\|\varepsilon + \tfrac{\varepsilon}{3}(q - 2\eta) < -\tfrac{2}{3}\|h''\|\varepsilon + \tfrac{\varepsilon}{3}(\|h'\| - \eta) \\
& \leq -\tfrac{2}{3}\|h''\|\varepsilon + \tfrac{\varepsilon}{3}\|h''\| \\
& = -\tfrac{\varepsilon}{3}\|h''\|.
\end{align*}
We have verified that \eqref{eq:frechet-subdif-eps3} holds which concludes the proof.
\end{proof}

\begin{remark}\rm One can compare the proof above with a similar proof using rich families. Note that the formula marked by $(*)$ above is similar to the formula which is needed in the inductive construction of the corresponding rich family; see \eqref{16} and the definition of $\widetilde{C}$ in the proof of Theorem \ref{nova}. When using the method of suitable models, this inductive construction is hidden in Theorem \ref{T:countable-model}.

More precisely, the proof of the theorem above says that, whenever a set $M$ is constructed in such a way that it is ``closed under formula $(*)$'' (and maybe also under some formulas needed in Lemmas \ref{L:unique-M} and \ref{l:basics-in-M}), then $\overline{X\cap M}$ will give us the needed separable subspace. And Theorem \ref{T:countable-model} guarantees that such $M$ can be constructed.
\end{remark}

Now, we get the following analogy of Theorem \ref{princip} as a combination of the previous results.

\begin{theorem}\label{princip-modely}There are a finite list of formulas $\Phi$ and a countable set $C$ such that any $M \prec (\Phi; C)$ satisfies the following:\\[5pt]
Let $k\in \N$, let $X$ be a non-separable {\tt Asplund} space, let $Z_1,\ldots, Z_k$ be Banach spaces, 
let $z^*_1\in Z^*_1,\ \ldots, z^*_k\in Z^*_k$, let $T_i: Z_i\to X,\ i=1,\ldots,k$,  
be bounded linear operators, and let $f:X\longrightarrow(-\infty,+\infty]$ be a proper function. Assume that $\{T_1,\ldots,T_k, f, Z_1, \ldots, Z_k, z^*_1, \ldots z^*_k, X\}\subset M$. Then $T_1((Z_1)_M)\subset X_M,\ \ldots,\ T_k((Z_k)_M)\subset X_M$, $\overline{X^*\cap M}$ is a (closed separable) subspace of $X^*$ and the following holds.
\item{(i)} The assignment $\overline{X^*\cap M}\ni x^*\longmapsto x^*|_{X_M}\in (X_M)^*$ is an isometry {\tt onto} $(X_M)^*$.
\item{(ii)} 
$
\big(\partial_Ff(x)\cap \overline{X^*\cap M}\big)|_{X_M}=\big(\partial_Ff(x)\big)|_{X_M}=\partial_F(f|_{X_M})(x)\,
$ for every $x\in X_M$; and
\item{(iii)} 
$\|T^*_ix^* - z^*_i\| = \|(T_i|_{(Z_i)_M})^*(x^*|_{X_M}) - (z^*_i|_{(Z_i)_M})\|$ for every $x^*\in \overline{X^*\cap M}$ and $i = 1,\ldots k$.
\end{theorem} 

\begin{proof}Let us fix a finite list of formulas $\Phi$ and a countable set $C$  which consist of all the formulas (sets) from the statements of Lemmas \ref{l:basics-in-M}, \ref{l:dualInM} and Theorems \ref{dno-modely}, \ref{dnoo-modely}, \ref{nova-modely} and fix any $M \prec (\Phi; C)$ with $\{T_1,\ldots,T_k, f, Z_1, \ldots, Z_k, z^*_1, \ldots z^*_k, X\}\subset M$. By Theorem \ref{dno-modely}, we have (i) and, by Lemma \ref{l:dualInM} and Theorem \ref{dnoo-modely}, it remains to verify (ii). Fix any $x\in X_M$ .  We obviously have that 
$$
\big(\partial_Ff(x)\cap \overline{X^*\cap M}\big)|_{X_M} \subset \big(\partial_Ff(x)\big)|_{X_M} \subset \partial_F(f|_{X_M})(x).
$$
It remains to prove that $\partial_F(f|_{X_M})(x) \subset \big(\partial_Ff(x)\cap \overline{X^*\cap M}\big)|_{X_M}$. If $y^*\in \partial_F(f|_{X_M})(x)$ we have, by (i), $x^*\in \overline{X^*\cap M}$ with $y^* = x^*|_{X_M}$ and, by Theorem \ref{nova-modely}, $x^*\in \partial_Ff(x)$; hence, $y^* = x^*|_{X_M}\in  \big(\partial_Ff(x)\cap \overline{X^*\cap M}\big)|_{X_M}$.
\end{proof}

\begin{remark}\rm Observe that Theorem \ref{princip-modely} actually provides a cofinal family of separable subspaces of $X$ satisfying Theorem \ref{princip}. (Indeed, it is enough to use Theorem \ref{T:countable-model} which says that, for every countable set $S$, it is possible to construct the set $M$ in such a way that $S\subset M$.) Using Theorem \ref{t:richIffModel} from the next section, we can even deduce the existence of a rich family of separable subspaces of $X$ with such properties. However, in some applications, we do not need to know a family is rich (cofinality is enough). The property of $\sigma$-closeness is usually needed if we want to combine several 
(but not more than countably many) statements together; see Proposition \ref{bm}. When using the method of suitable models, we are able to combine finitely many statements by putting together finitely many  lists of formulas.
\end{remark}

\section{Rich families versus suitable models}\label{section5}

As mentioned before, there are two methods of proving separable reduction theorems. One is ``the method of rich families'' and the other one is ``the method of suitable models''. 
The relation between those two methods was investigated in \cite{ck}, where the authors proved that, if there is a proof using the method of rich families, there is also a proof using the method of suitable models. In many cases those two methods are equivalent; however, it is not known to the authors whether they are equivalent in general.

In this section we prove that both methods are in some sense equivalent in Asplund spaces and in the spaces isomorphic to the space $C(K)$ of continuous functions over some zero-dimensional compact space $K$; see Theorem \ref{t:richIffModel}. (Recall that a compact space is zero-dimensional if it has a basis consisting of clopen sets --- i.e. sets which are both closed and open.) This gives a partial positive answer to \cite[Question 2.8]{ck} and \cite[Question 3.6]{ck}.

The main reason for investigating such a relation between those two methods is that it gives mathematicians the freedom to use the method they prefer and make the results applicable by another mathematician preferring the other method. For example, once we know those two methods are equivalent, we can prove a certain statement using the method of suitable models and formulate the resulting theorem in the language of rich families --- hence not using any terminology from logic or set theory and making it more applicable by non experts. We refer a reader to the text after Theorem \ref{t:richIffModel}, where we show a concrete example of how one may deduce (i)$\implies$(iii) in Theorem \ref{dno} from Theorem \ref{dno-modely} and how to deduce a slightly weaker version of Theorem \ref{dno-modely} from (i)$\implies$(iii) in Theorem \ref{dno}.

In order to compare both methods of separable reduction, we use the definition from \cite[Section 3]{ck}. Since it is quite long and intuitively clear notion and we will actually not use the definition here, we do not explicitly articulate it. However, what we will use is the following notion.

\begin{definition}\label{d:richNice}\rm Let $X$ be a Banach space. We say that \emph{suitable models generate nice rich families in $X$}, if the following holds:\\[3pt]
Whenever $Y$ is a countable set and $\Phi$ is a finite list of formulas, there exists a family $\M$ satisfying the following conditions:
\begin{enumerate}[\upshape (i)]
	\item For every $M\in\M$, $M\prec(\Phi;Y)$.
	\item The set $\{X_M:\; M\in\M\}$ is a rich family of separable subspaces in $X$.
	\item $\forall M,N\in\M,\quad M\subset N \Longleftrightarrow \overline{X\cap M}\subset \overline{X\cap N}$.
\end{enumerate}
\end{definition}

The following result follows immediately from the proof of \cite[Theorem 3.2]{ck}; for the definitions of notions used in (i) and (ii) below we refer to \cite[Section 3]{ck}.

\begin{theorem}\label{tRichAModel}Let $X$ be a Banach space, $A\subset X$ and $f$ a function with $\operatorname{Dom}(f)\subset X$. Let $\phi(X,A,f,C_1,\ldots,C_k)$ be a statement concerning the Banach space $X$, the set $A$, the function $f$ and constants $C_1,\ldots,C_k$. Consider the following conditions\vspace{1mm}
	\item {(i)} $\phi(X,A,f,C_1,\ldots,C_k)$ is separably determined by the method of suitable models.
	\item {(ii)} $\phi(X,A,f,C_1,\ldots,C_k)$ is separably determined by the method of rich families.\vspace{1mm}
\noindent Then (ii) implies (i). Moreover, if suitable models generate nice rich families in $X$, then both conditions are equivalent.
\end{theorem}

It is proved in \cite[Theorem 2.7]{ck} that suitable models generate nice rich families in $X$ whenever $X$ has a fundamental minimal system or $\operatorname{dens}X = \aleph_1$. (Recall that a family $\Gamma$ of vectors in $X$ is called a {\it fundamental minimal system}, if $\overline{\rm sp}\,\Gamma=X$ and, for every $\gamma\in\Gamma$, we have $\gamma\notin\overline{\rm sp}\, (\Gamma\setminus\{\gamma\}).\big)$ It is not known to the authors whether there exists a Banach space $X$ such that suitable models do not generate nice rich families in $X$. The main results of this section are the following.

\begin{theorem}\label{t:AsplundRichModely}Let $X$ be any Asplund space. Then suitable models generate nice rich families in $X$. Moreover, the family $\M$ from the Definition \ref{d:richNice} is such that $\big\{\overline{X^*\cap M}:\; M\in \M\big\}$ is a rich family of separable subspaces in $X^*$ and, for every $M$, $N\in \M$, we have $X_M\subset X_N$ if and only if $\overline{X^*\cap M}\subset\overline{X^*\cap N}$.
\end{theorem}

\begin{theorem}\label{t:CompactRichModely}Let $K$ be any zero-dimensional compact space. Then suitable models generate nice rich families in $X$ whenever $X$ is isomorphic to $C(K)$.
\end{theorem}

The latter statement gives a partial positive answer to \cite[Question 2.8]{ck}. As a consequence, using Theorem \ref{tRichAModel}, we immediately obtain a partial positive answer to \cite[Question 3.6]{ck}.

\begin{theorem}\label{t:richIffModel}Let $X$, $A$, $f$, $C_1,\ldots,C_k$ and $\phi(X,A,f,C_1,\ldots,C_k)$ be as in Theorem \ref{tRichAModel}. If $X$ is Asplund or isomorphic to $\CC(K)$, where $K$ is zero-dimensional compact, then the following conditions are equivalent:
	\item {(i)} $\phi(X,A,f,C_1,\ldots,C_k)$ is separably determined by the method of suitable models.
	\item {(ii)} $\phi(X,A,f,C_1,\ldots,C_k)$ is separably determined by the method of rich families.
\end{theorem}

Before proving Theorems \ref{t:AsplundRichModely} and \ref{t:CompactRichModely}, let us discuss their
content. Assume that $X$ is an Asplund space. If we deal with rich families of separable subspaces in $X$, then Theorem \ref{t:richIffModel} says that both methods are equivalent; see e.g. \cite[Corollaries 3.3 and 3.4]{ck} for some examples. However, if we are interested in the relation with rectangle-families, the situation is more complicated and we have to deal with more details. For example, let us examine the relation between Theorem \ref{dno} (i)$\implies$(iii) and Theorem \ref{dno-modely}.

``Theorem \ref{dno}$\implies$Theorem \ref{dno-modely}'': We can obtain a bit weaker version of Theorem \ref{dno-modely} from (i)$\implies$(iii) in Theorem \ref{dno}. The weakening is in the fact that our model will depend on the Banach space $X$. However, this is not important in applications and so we can say that ``the method of rectangle-families implies the method of suitable models'' in this case. Let $X$ be a Banach space, $C$ a countable set and $\Phi$ a finite list of formulas. Then there is another countable set $D\supset C$ and a finite list of formulas $\Psi$, containing the formulas of $\Phi$, such that whenever $M\prec(\Psi;D)$, the assignment $\overline{X^*\cap M}\ni x^*\longmapsto x^*|_{X_M}\in (X_M)^*$ is an isometry onto $(X_M)^*$. Indeed, by \cite[Proposition 3.1]{ck}, we can find a set $D\supset C$ and a finite list of formulas $\Psi$, containing all formulas from $\Phi$, such that for every $M\prec(\Psi;D)$ we have $\overline{(X\times X^*)\cap M}\in\AA$, where $\AA\subset \SS(X\times X^*)$ is the rich family from (iii) in Theorem \ref{dno}. Now, it is quite easy to check that by enlarging the list $\Psi$, we may assume that $\overline{(X\times X^*)\cap M} = X_M\times \overline{X^*\cap M}$. Hence, $\overline{X^*\cap M}\ni x^*\longmapsto x^*|_{X_M}$ is an isometry onto $(X_M)^*$.

``Theorem \ref{dno-modely}$\implies$ Theorem \ref{dno}'': We can obtain exactly (i)$\implies$ (iii) in Theorem \ref{dno} from Theorem \ref{dno-modely}. Indeed, by Theorems \ref{dno-modely} and \ref{t:AsplundRichModely}, there is a family $\M$ of sets such that both $\{X_M:\; M\in\M\}$ and $\{\overline{X^*\cap M}:\; M\in\M\}$ are rich families of separable subspaces in $X$ and $X^*$ respectively, $\overline{X^*\cap M}\ni x^*\longmapsto x^*|_{X_M}$  is an isometry onto $(X_M)^*$ for every $M\in\M$ and whenever $M$, $N\in M$, we have $X_M\subset X_N$ if and only if $\overline{X^*\cap M}\subset \overline{X^*\cap N}$. Thus, the desired rectangle-family
is $\{X_M\times \overline{X^*\cap M}:\; M\in\M\}$.

Of course, we could similarly show that any other results concerning an Asplund space $X$ and its dual formulated either in the language of rectangle-families on $X\times X^*$ or in the language of suitable models are in some sense equivalent. However, formulating such a result in a precise way requires introducing a rather complicated definition and it seems to us that such a result is not worth it.\\

In the remainder of this section we prove Theorems \ref{t:AsplundRichModely} and \ref{t:CompactRichModely}. 
We denote by $[I]^{\leq\omega}$ the family of all countable subsets of $I$ (unlike of Section 2). 
The core of the problem is to find a class of models for which we have $X_M\subset X_N$ (if and) only if $M\subset N$; see condition (iii) in Definition \ref{d:richNice}.

First, let us consider the case of Asplund spaces. There, using some observations from \cite{ck}, we know that the dual space admits such a class of models. An argument similar to 
the proof of the implication (ii)$\implies$(iii) in Theorem \ref{dno} gives us the possibility to move from the space to its dual, where the construction is already known. This is expressed more precisely in Lemma \ref{l:goDown}. We begin with a preliminary result.

\begin{lemma}\label{l:skrinka}There are a finite list of formulas $\Phi$ and a countable set $C$ such that the any $M \prec (\Phi; C)$ satisfies the following:\\[5pt]
Let $X$ be an Asplund space with an Asplund generator $G$ and assume that $\{X, G\}\subset M$. Then $\overline{X^*\cap M} = \overline{G(X\cap M)}$.
\end{lemma}

\begin{proof} Take the list $\Phi$ and the countable set $C$  from Lemma \ref{l:basics-in-M}. We may assume that $C\supset \Q$. 
Append to $\Phi$ the formula $(*)$ marked below and all subformulas of it and denote such an enhanced
list again by $\Phi$.
Fix any $M \prec (\Phi; C)$, with $\{X,G\}\in M$.

For every $x\in X\cap M$, by Lemma \ref{l:basics-in-M} (i) (iii), we have $G(x)\subset X^*\cap M$; hence $G(X\cap M)\subset X^*\cap M$ and $\overline{G(X\cap M)}\subset \overline{X^*\cap M}$. For the other inclusion, fix 
any $x^*\in X^*\cap M$. Then, for every $n\in\N$, by the property (c) of Asplund generator,
\[\tag{$*$}
\exists C\in [X]^{\leq\omega} \ \exists y^*\in G(C):\;\|x^* - y^*\|\leq \tfrac{1}{n}\,.
\]
By Lemma \ref{L:unique-M}, using the absoluteness of this formula and its subformula, for every $n\in\N$
we find $C_n\in [X]^{\leq\omega}\cap M$ and $y^*_n\in G(C_n)$ satisfying the formula above (note that, for every $n\in\N$, we use the absoluteness of one formula --- what varies is the parameter $n$).  By Lemma \ref{l:basics-in-M} (iii), $C_n\subset X\cap M$; hence, by the property (b) of Asplund generator, $y^*_n\in G(X\cap M)$ for every $n\in\N$ and we get $x^*\in \overline{G(X\cap M)}$. Thus, $X^*\cap M\subset \overline{G(X\cap M)}$. It follows that $\overline{X^*\cap M}\subset \overline{G(X\cap M)}$.
\end{proof}

\begin{lemma}\label{l:goDown}There are a finite list of formulas $\Phi$ and a countable set $C$ such that the following holds:\\[5pt]
Let $X$ be an Asplund space with an Asplund generator $G$ and let $M, N \prec (\Phi; C)$ be such that $\{X, G\}\subset M\cap N$. If $\overline{X\cap M}\subset \overline{X\cap N}$, then $\overline{X^*\cap M}\subset \overline{X^*\cap N}$.
\end{lemma}
\begin{proof}Let $\Phi$ and $C$ be the unions of all the lists $\Phi$'s and the sets $C$'s 
from Lemmas \ref{l:basics-in-M}, \ref{l:dualInM} and \ref{l:skrinka}, respectively. Fix any $M,N \prec (\Phi; C)$ 
such that $\{X,G\}\subset M\cap N$ and $\overline{X\cap M}\subset \overline{X\cap N}$.

Pick a countable dense subset $D\subset \overline{X\cap N}$ with $D\supset X\cap M$. Then we have, using Lemma \ref{l:skrinka}, the properties (b) and (d) of Asplund generator, and the fact that $\overline{X^*\cap N}$ is a closed subspace,
$$\overline{X^*\cap M} = \overline{G(X\cap M)}\subset \overline{\rm sp}\, {G(D)} = \overline{\rm sp}\, {G(X\cap N)} = \overline{X^*\cap N},$$ 
which completes the proof.
\end{proof}


\begin{proof}[Proof of Theorem \ref{t:AsplundRichModely}]We know from, e.g., \cite[page 166]{hmvz}, that $X^*$ admits a Markushevich basis. Then the ``bottom'' of this basis will clearly be a fundamental minimal system; call it $\Gamma$. Note that $|\Gamma| = \operatorname{dens} X^* = \operatorname{dens} X$. Hence, we can pick a subset $\{d_\gamma:\ \gamma\in\Gamma\}$ dense in $X$ and define the mapping $H:[\Gamma]^{\leq\omega}
\longrightarrow[X]^{\leq\omega}$ by $H(A):=\{d_\gamma:\; \gamma\in A\},\ A\in[\Gamma]^{\leq\omega}$. By Theorem \ref{dno}, pick an Asplund generator $G$ in $X$. In order to see that suitable models generate nice rich families in $X$, fix a countable set $Y$ and a finite list of formulas $\Phi$. We may without loss of generality assume that  $\{X,H,G\}\subset Y$ and that the set $Y$ and the list $\Phi$ contain the sets $C$'s and the lists $\Phi$'s
from the statements of Lemmas \ref{l:goDown}, \ref{l:basics-in-M}, respectively (because the condition (i) in Definition \ref{d:richNice} is inherited by countable subsets and shorter lists of formulas).

Since $X^*$ admits a fundamental minimal system, it follows immediately from the proof of \cite[Lemma 2.5]{ck} that there exists a family $\M: = \{M_A:\; A\in[\Gamma]^{\leq\omega}\}$ satisfying the following conditions.
	\item{(i')} For every $A\in[\Gamma]^{\leq\omega}$, $M_A\prec (\Phi; Y)$ and $A\subset M_A$.
	\item{(ii')} For every $A, B\in [\Gamma]^{\leq\omega}$, $\overline{X^*\cap M_A}\subset \overline{X^*\cap M_B}$ (if and) only if $M_A\subset M_B$.
	\item{(iii')} For every increasing sequence $(M_n)$ from $\M$ we have $\bigcup M_n\in\M$.

We will show that the family $\M$ satisfies conditions (i)--(iii) from the Definition \ref{d:richNice}. Condition (i) follows from (i') above. In order to verify (iii), observe that, for every $A,B\in [\Gamma]^{\leq\omega}$, we have
$$\overline{X\cap M_A}\subset \overline{X\cap M_A}\stackrel{\text{Lemma \ref{l:goDown}}}{\implies} \overline{X^*\cap M_A}\subset \overline{X^*\cap M_A}\stackrel{\text{(ii')}}{\implies} M_A\subset M_B.$$
Hence, condition (iii) holds. It remains to show (ii), that is, that $\FF := \{X_M:\; M\in\M\}$ is a rich family of separable subspaces in $X$

As for the $\sigma$-closeness property of $\FF$, let $X_{M_1}\subset X_{M_2}\subset\cdots$ be an increasing sequence of separable subspaces from $\FF$. By the already verified condition (iii), we have that $(M_n)$ is an increasing sequence. By (iii'), we get $\bigcup M_n \in \M$. Now, it is easy to verify (see e.g. \cite[Lemma 3.4]{c}) that $\overline{\bigcup X_{M_n}} = X_{\bigcup M_n}\in \FF$.

In order to show that $\FF$ is cofinal, we use the mapping $H$. Let $C$ be a countable set in $X$. Then there is a countable subset $A$ of $I$ with $C\subset \overline{H(A)}$. Hence, by (i') and Lemma \ref{l:basics-in-M}, we have $H(A)\subset M_A$ and $C\subset \overline{X\cap M_A} = X_{M_A}\in\FF$.
\end{proof}

Now, let us consider the case of $C(K)$ spaces where $K$ is zero-dimensional compact. In order to prove  Theorem \ref{t:CompactRichModely}, we first observe the following.

\begin{lemma}\label{l:Isomorphism}Let $X$ and $Z$ be isomorphic Banach spaces. If suitable models generate nice rich families in $X$, then suitable models generate nice rich families in $Z$.
\end{lemma}

\begin{proof}Let $T:X\to Z$ be an isomorphism onto. Fix a countable set $Y$ and a finite list of formulas $\Phi$. We may without loss of generality assume that $X$, $T$, $T^{-1}$, $Z$ and all the sets from the statement of Lemma \ref{l:basics-in-M} are elements of $Y$ and that our $\Phi$ contains all the formulas from the statement of Lemma \ref{l:basics-in-M}.

Let $\M$ be a family satisfying Definition \ref{d:richNice} (i)--(iii) for the space $X$. We will show that this family is sufficient for the space $Z$ as well. First, notice that $T(X_M) = Z_M$. Indeed, by Lemma \ref{l:basics-in-M}, we have $T(X\cap M)\subset Z\cap M$ and $T^{-1}(Z\cap M)\subset X\cap M$; hence, $T(X\cap M) = Z\cap M$ and since $T$ is isomorphism, $T(X_M) = Z_M$. Thus, whenever $M$, $N\in \M$, we have
$$M\subset N \Longleftrightarrow \overline{X\cap M}\subset \overline{X\cap N} \Longleftrightarrow \overline{Z\cap M}\subset \overline{Z\cap N}$$
and it remains to show that $\{Z_M:\; M\in\M\}$ is a rich family of separable subspaces in $Z$. This easily follows from the above and the fact that $\{X_M:\; M\in\M\}$ is a rich family of separable subspaces in $X$.
\end{proof}

\begin{proof}[Proof of Theorem \ref{t:CompactRichModely}]Let $K$ be a zero-dimensional compact. By Lemma \ref{l:Isomorphism}, we may assume that $X = C(K)$. Denote by $\operatorname{Clop} K$ the set of clopen sets in $K$. For $A\in \operatorname{Clop} K$, we denote by $\iota_A$ the indicator function of $A$; i.e. the 
(continuous) function defined by $\iota_A(x): = 1$ if $x\in A$ and $\iota_A(x): = 0$ otherwise. Using the Stone--Weierstrass theorem, it is easy to see that $\overline{\rm sp}\, \{\iota_A:\; A\in\operatorname{Clop} K\} = C(K)$. Hence, by \cite[Lemma 2.5]{ck}, it suffices to show there is a countable set $C$ and a finite list of formulas $\Phi$ such that for every $M\prec (\Phi; C)$ and every $B\in \operatorname{Clop} K$,
$$
\iota_B \in \overline{\rm sp}\, \{\iota_A:\; A\in\operatorname{Clop} K\cap M\}\; \implies\; B\in M.
$$

Fix a finite list of formulas $\Phi$ and a countable set $C$  from the statement of Lemma \ref{l:basics-in-M}. We may assume that $C\supset\Q$. Fix any $M \prec (\Phi; C)$ and any $B\in \operatorname{Clop} K$ with $\iota_B \in \overline{\rm sp}\, \{\iota_A:\; A\in\operatorname{Clop} K\cap M\}$. Then there is $f\in {\rm sp}\, \{\iota_A:\; A\in\operatorname{Clop} K\cap M\}$ with $\|\iota_B - f\| < \frac{1}{4}$. By Lemma \ref{l:basics-in-M}, $M$ is closed under finite intersections and differences of sets; hence, we can easily verify (e.g. by induction) that $f$ can be written as a linear combination of indicator functions of pairwise disjoint sets from $\operatorname{Clop} K\cap M$. Since $\|\iota_B - f\| < \frac{1}{4}$, we have that $B$ is the union of finitely many sets from $\operatorname{Clop} K\cap M$; hence, by Lemma \ref{l:basics-in-M}, $B\in M$.
\end{proof}

\begin{remark}\rm 
Let $\Gamma$ be a set of cardinality greater than continuum and let $\ell_\infty^c(\Gamma)$ denote the closed subspace of $\ell_\infty(\Gamma)$ consisting of all vectors with countable support.
This space does not have a fundamental minimal system and is of density greater than $\aleph_1$. However, it is isomorphic to the space $C(K)$, where $K$ is the zero-dimensional compact space constructed from the Boolean algebra of subsets of $\Gamma$ which are either countable or have countable complement.
Theorem \ref{t:CompactRichModely} applies to the space $\ell_\infty^c(\Gamma)$, and thus we
get a partial positive answer to \cite[Question 2.8]{ck}. 
\end{remark}

\end{document}